\documentclass{article}

\usepackage{graphicx}
\usepackage{hyperref}
\usepackage{amsfonts, amsmath, amsthm}
\usepackage[margin=3cm]{geometry}
\usepackage{tikz}
\usetikzlibrary{calc, arrows, decorations.markings, decorations.pathmorphing, positioning, decorations.pathreplacing}
\makeatletter
\tikzset{nomorepostaction/.code={\let\tikz@postactions\pgfutil@empty}}
\makeatother

\usepackage{fancyhdr}
\newtheorem{thm}{Theorem}[section]
\newtheorem{cor}[thm]{Corollary}
\newtheorem{lem}[thm]{Lemma}
\newtheorem{prop}[thm]{Proposition}

\newtheorem{defn}[thm]{Definition}

\newcommand{\To}{\longrightarrow}

\newcommand{\Z}{\mathbb{Z}}

\newcommand{\0}{{\bf 0}}
\newcommand{\1}{{\bf 1}}

\DeclareMathOperator{\Byp}{Byp}

\DeclareMathOperator{\IIm}{Im}

\begin{document}

\title{Dimensionally-reduced sutured Floer homology as a string homology}

\author{Daniel V. Mathews and Eric Schoenfeld}%

\date{}

\maketitle

\begin{abstract}
We show that the sutured Floer homology of a sutured 3-manifold of the form $(D^2 \times S^1, F \times S^1)$ can be expressed as the homology of a string-type complex, generated by certain sets of curves on $(D^2, F)$ and with a differential given by resolving crossings. We also give some generalisations of this isomorphism, computing ``hat'' and ``infinity'' versions of this string homology. In addition to giving interesting elementary facts about the algebra of curves on surfaces, these isomorphisms are inspired by, and establish further, connections between invariants from Floer homology and string topology.
\end{abstract}

\tableofcontents

\section{Introduction}

On the one hand, this paper is about an interesting combinatorial/topological fact about curves on surfaces. On the other hand, it establishes some connections between invariants of $3$-manifolds from contact topology, Floer homology, and string topology.

\subsection{A combinatorial question about curves on surfaces}

We consider the following simple question. Fix an oriented surface $\Sigma$ and a finite set of signed points $F$ on $\partial \Sigma$. Consider sets $s$ of immersed curves on $\Sigma$ with $\partial s = F$. These sets of curves, which we call \emph{string diagrams}, consist of immersed closed curves and immersed arcs with boundary points on $F$. Take the $\Z_2$ vector space spanned by homotopy classes of string diagrams. (Several meanings of ``homotopy'' are possible here, as we will see.) On this vector space, there is a \emph{differential} $\partial$ defined by \emph{resolving crossings}, as shown:

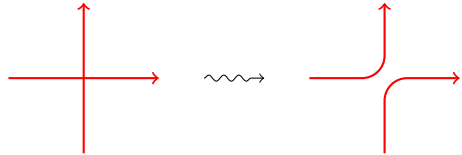
\begin{figure}[h]
\begin{center}
\begin{tikzpicture}[scale=1, string/.style={thick, draw=red, -to}]

\draw [string] (-1,0) -- (1,0);
\draw [string] (0,-1) -- (0,1);

\draw [shorten >=1mm, -to, decorate, decoration={snake,amplitude=.4mm, segment length = 2mm, pre=moveto, pre length = 1mm, post length = 2mm}]
(1.5,0) -- (2.5,0);

\draw [string] (3,0) -- (3.7,0) to [bend right=45] (4,0.3) -- (4,1);
\draw [string] (4,-1) -- (4,-0.3) to [bend left=45] (4.3,0) -- (5,0);
\end{tikzpicture}
\caption{Resolving a crossing}
\label{fig:resolving_crossing}
\end{center}
\end{figure}

One can show that, with appropriate definitions of the words above, this is a chain complex, and hence has a homology, which we call \emph{string homology}. (Notwithstanding other uses of this word: \cite{CS}.) The question is: What is the homology? We will give some answers to this question for two variants of the definitions --- which we shall define in due course, and which we shall argue are the only variants for which the question is a reasonable one. The two chain complexes will be called $\widehat{CS}(\Sigma,F)$ and $CS^\infty(\Sigma,F)$, and their homologies $\widehat{HS}(\Sigma,F)$ and $HS^\infty(\Sigma,F)$.

The answer appears to be that (i) homology is zero unless $F$ is \emph{alternating}, i.e. the points alternate in sign around $\partial \Sigma$; (ii) any element of homology can be represented by string diagrams which are \emph{sets of sutures}; and (iii) the ``only'' relation between sets of sutures in this homology is the \emph{bypass relation} introduced by Honda--Kazez--Mati\'{c} \cite{HKM08} and developed by the first author \cite{Me09Paper, Me10_Sutured_TQFT, Me11_torsion_tori, Me12_itsy_bitsy}, shown in figure \ref{fig:bypass_relation}.

\begin{figure}[h]
\begin{center}

\begin{tikzpicture}[
scale=1.2, 
string/.style={thick, draw=red, postaction={nomorepostaction, decorate, decoration={markings, mark=at position 0.5 with {\arrow{>}}}}}]

\draw (3,0) circle (1 cm); 	
\draw (0,0) circle (1 cm);
\draw (-3,0) circle (1 cm);

\draw [string] (30:1) arc (120:240:0.57735);
\draw [string] (0,-1) -- (0,1);
\draw [string] (150:1) arc (60:-60:0.57735);

\draw [string] (3,0) ++ (150:1) arc (-120:0:0.57735);
\draw [string] (3,0) ++ (30:1) -- ($ (3,0) + (210:1) $);
\draw [string] (3,0) ++ (-90:1) arc (180:60:0.57735);

\draw [string] (-3,0) ++ (30:1) arc (-60:-180:0.57735);
\draw [string] ($ (-3,0) + (150:1) $) -- ($ (-3,0) +  (-30:1) $);
\draw [string] (-3,-1) arc (0:120:0.57735);

\draw (-1.5,0) node {$+$};
\draw (1.5,0) node {$+$};
\draw (4.5,0) node {$=0$};

\end{tikzpicture}
\caption{Bypass relation.}
\label{fig:bypass_relation}
\end{center}
\end{figure}
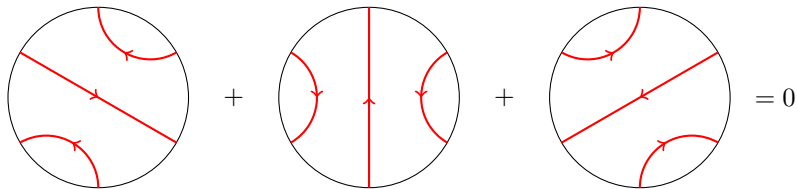

The simplest illustration of the ``reason'' for this relation (far from a proof, of course) is figure \ref{fig:fundamental_boundary}.

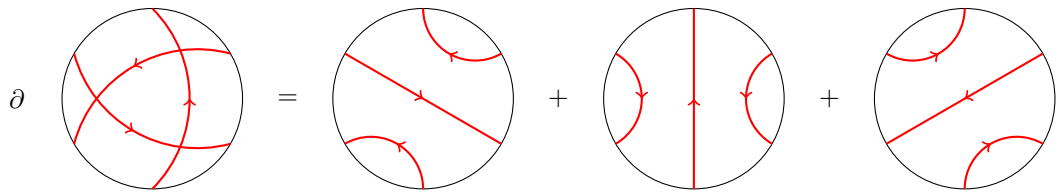
\begin{figure}[ht]
\begin{center}
\begin{tikzpicture}[
scale=1.2, 
string/.style={thick, draw=red, postaction={nomorepostaction, decorate, decoration={markings, mark=at position 0.5 with {\arrow{>}}}}}]

\draw (-6,0) circle (1 cm);
\draw (3,0) circle (1 cm); 	
\draw (0,0) circle (1 cm);
\draw (-3,0) circle (1 cm);

\draw [string] ($ (-6,0) + (-90:1) $) to [bend right=45] ($ (-6,0) + (90:1) $);
\draw [string] ($ (-6,0) + (30:1) $) to [bend right=45] ($ (-6,0) + (210:1) $);
\draw [string] ($ (-6,0) + (150:1) $) to [bend right=45] ($ (-6,0) + (-30:1) $);

\draw [string] (30:1) arc (120:240:0.57735);
\draw [string] (0,-1) -- (0,1);
\draw [string] (150:1) arc (60:-60:0.57735);

\draw [string] (3,0) ++ (150:1) arc (-120:0:0.57735);
\draw [string] (3,0) ++ (30:1) -- ($ (3,0) + (210:1) $);
\draw [string] (3,0) ++ (-90:1) arc (180:60:0.57735);

\draw [string] (-3,0) ++ (30:1) arc (-60:-180:0.57735);
\draw [string] ($ (-3,0) + (150:1) $) -- ($ (-3,0) +  (-30:1) $);
\draw [string] (-3,-1) arc (0:120:0.57735);

\draw (-7.5,0) node {$\partial$};
\draw (-4.5,0) node {$=$};
\draw (-1.5,0) node {$+$};
\draw (1.5,0) node {$+$};

\end{tikzpicture}
\caption{The bypass relation is a boundary.}
\label{fig:fundamental_boundary}
\end{center}
\end{figure}

In this paper we will prove the above results when $\Sigma$ is a disc $D^2$; and some partial results for $\Sigma$ a general surface.
\begin{thm}
\label{thm:zero_thm}
\label{thm:non_alternating_homology}
If $F$ does not alternate in sign then $\widehat{HS}(\Sigma,F) = HS^\infty(\Sigma,F) = 0$.
\end{thm}

\begin{thm}
\label{thm:rough_main_thm1}
If $F$ does alternate in sign, then
\[
\widehat{HS}(D^2 ,F) \cong \frac{ \Z_2 \langle \text{isotopy classes of sutures on $(D^2, F)$} \rangle }{ \text{Bypass relation} }
\]
and $HS^\infty(D^2, F) \cong \Z_2[U, U^{-1}] \otimes \widehat{HS}(D^2, F)$.
\end{thm}

As the notation here suggests, the ``$\infty$-complex'' $CS^\infty$ has a ``$U$-map'' and indeed the notation is by analogy with Floer homology; see sections \ref{sec:string_complex} and \ref{sec:U_map} below. We will discuss the various structures preserved by these isomorphisms as we proceed.

\subsection{Relations between Floer-theoretic invariants}
\label{sec:floer-theoretic}

The above combinatorial question about curves on surfaces is in fact motivated by relations between several invariants of $3$-manifolds. In particular, it sheds light on the correspondence between $SFH$, sutured Floer homology, and $ECH$, embedded contact homology, for $3$-manifolds with sutured boundary of the form $(\Sigma \times S^1, F \times S^1)$, where $\Sigma$ is a surface with boundary and $F$ is a finite subset of $\partial \Sigma$. Because such a $3$-manifold is then ``reduced'' to a product of a surface and a circle, this can be considered as a \emph{dimensionally-reduced} version of $SFH$.

It has long been believed that several Floer-theoretic invariants associated to a closed oriented 3-manifold $M$ are equivalent: Heegaard Floer homology $HF$ (as defined by Ozsvath--Szabo beginning in \cite{OS04Closed, OS04Prop}), embedded contact homology $ECH$ (as defined by Hutchings in \cite{Hutchings02}), and Seiberg-Witten Floer homology $HM$ (as defined by Kronheimer--Mrowka in \cite{Kronheimer_Mrowka07}). All of these invariants come in various flavours, for instance $\widehat{HF}$, $HF^+$, $HF^-$, $HF^\infty$. Recent work of Kutluhan--Lee--Taubes, in a series of $5$ papers running to well over $750$ pages \cite{KLT10_HF_HM_1, KLT10_HF_HM_2, KLT10_HF_HM_3, KLT11_HF_HM_4, KLT12_HF_HM_5}, asserts a proof of these equivalences, in their various flavours, for closed connected oriented 3-manifolds. Independent work of Colin--Ghiggini--Honda \cite{CGH10}, also running into hundreds of pages (announced in \cite{CGH10}, summarised in \cite{CGH_HF_ECH_summary}, detail in \cite{CGH_ECH_open_books, CGH_HF_ECH_1, CGH_HF_ECH_2, CGH_HF_ECH_3}), asserts a proof of similar equivalences between $HF$ and $ECH$, avoiding Seiberg-Witten theory and using open books.

These correspondences are expected to apply, with appropriate modifications, also to 3-manifolds with boundary. Although there is a theory of bordered Heegaard Floer homology for general 3-manifolds with boundary \cite{LOT08}, the more restricted class of \emph{sutured} 3-manifolds plays a natural role in both Heegaard Floer homology and embedded contact homology. Sutured Floer homology $SFH$, as defined by Juh\'{a}sz \cite{Ju06}, is a generalisation of $\widehat{HF}$ to sutured 3-manifolds. Analogously, Colin--Ghiggini--Honda--Hutchings \cite{CGHH} have given a definition of $ECH$ for sutured 3-manifolds.

This paper explores the $HF$--$ECH$ correspondence, in a combinatorial form, in the particular case of sutured manifolds of the product form $(\Sigma \times S^1, F \times S^1)$, where $\Sigma$ is a compact oriented surface with nonempty boundary. Previous work of the authors and others has considered these two (or, in the case of $ECH$, only similar) homology theories in a combinatorial form, and in this paper we show these two combinatorial forms are related.

As regards $SFH$, in a series of papers \cite{Me09Paper, Me10_Sutured_TQFT, Me11_torsion_tori, Me12_itsy_bitsy}, the first author gave several combinatorial descriptions of various $SFH(\Sigma \times S^1, F \times S^1)$ in terms of a diagrammatic calculus of \emph{chord diagrams} or more generally \emph{sutures}. 

On the $ECH$ side, recall that  $ECH$ of a $3$-manifold is constructed by choosing a contact structure and counting certain holomorphic curves in the symplectization.  When that contact $3$-manifold is closed, there is another holomorphic curve theory, called Symplectic Field Theory, due to Eliashberg, Givental and Hofer, developed in \cite{EGH}.  Consider the special case when the $3$-manifold is the unit cotangent bundle of a hyperbolic surface $UT^* \Sigma$.  In \cite{CL}, Cieliebak and Latschev proved that the only relevant nontrivial holomorphic curves in this context correspond to \emph{resolving intersections between geodesics} on $\Sigma$, as in figure \ref{fig:resolving_crossing}.  In \cite{Goldman86, Turaev91}, Goldman and Turaev discovered the structures of a Lie bracket and cobracket on the space of geodesics by considering the same resolutions of intersection points.  All of the $SFT$ invariants in this case can then be described in terms of these combinatorial operations on the base $\Sigma$.  Indeed this is expected to be just a special case of a general relation between the $SFT$ invariants of $T^* M$ for any closed manifold $M$, and the string topology of $M$, as discovered by Chas and Sullivan in \cite{CS}.

Relating back to $ECH$, the holomorphic curves considered in this special case include those counted in the $ECH$ of the unit cotangent bundle $UT^* \Sigma$, given a very particular choice of contact and complex structure.  Thus one might expect the $ECH$ to be expressible in terms of these same combinatorial operations on geodesics.

Based on the above, we might make the following plausibility argument. When $\Sigma$ is a compact oriented surface with boundary, $UT^* \Sigma \cong \Sigma \times S^1$, and taking ``vertical'' sutures $\Gamma = F \times S^1 \subset \partial \Sigma \times S^1$ (where $F \subset \partial \Sigma$ is a finite set of signed points) gives a reasonable boundary structure for Reeb chords and corresponding holomorphic curves, which following \cite{CGHH} would require Reeb chords to flow in or out of $\Sigma$ according to the signed components of $\partial \Sigma \backslash F$. The $ECH$ complex might then, from the discussion above, be generated by homotopy classes of collections of curves in $\Sigma$ (i.e. \emph{string diagrams}), with a differential related to the Goldman bracket. On the other hand, according to the first author's work, the $SFH$ complex should be generated by isotopy classes of sutures on $(\Sigma, F)$, modulo a bypass relation. Thus there should be an isomorphism between the homology of a complex generated by curves on $\Sigma$ with a differential determined by resolving crossings, and a vector space of sutures on $(\Sigma, F)$ modulo a bypass relation, which is also $SFH(\Sigma \times S^1, F \times S^1)$. 

The theorems of this paper confirm some of these ideas, and we have the following result for discs.
\begin{thm}
\label{thm:rough_main_thm2}
Let $(D^2, F)$ be a sutured background disc. Then, with $\Z_2$ coefficients,
\[
\widehat{HS}(D^2, F) \cong SFH(D^2 \times S^1, F \times S^1).
\]
\end{thm}

The first author has shown that $SFH(\Sigma \times S^1, F \times S^1)$ is isomorphic to the $\Z_2$-vector space generated by isotopy classes of \emph{sets of sutures} on $(\Sigma, F)$, modulo the bypass relation. (This was shown in \cite{Me09Paper} for $\Sigma=D^2$ and for general surfaces follows immediately combining results there with a theorem of Juh\'{a}sz \cite{Ju08}. It also follows immediately from results of \cite{Me12_itsy_bitsy}.) From these results, theorem \ref{thm:rough_main_thm2} follows immediately from theorem \ref{thm:rough_main_thm1}.

The chain complexes $\widehat{CS}(\Sigma,F)$ and $CS^\infty(\Sigma,F)$ considered here possess several natural gradings, and our isomorphisms preserve some of them. Sets of sutures have an \emph{Euler class}, and under our isomorphisms, this grading corresponds to the grading by \emph{spin-c structure} in sutured Floer and Heegaard Floer homology. String diagrams are also naturally filtered by (minimal) \emph{number of intersections} between curves in a diagram; the differential decreases number of intersections. Since, as described above, all homology classes in string homology are represented by sutures, which have no intersections, all homology is carried in intersection-filtration level $0$.

The $\Z_2$-vector space $SFH(\Sigma \times S^1, F \times S^1)$ can also be described as a tensor power of a fundamental two-dimensional vector space, which in \cite{Me12_itsy_bitsy} was given basis $\{\0, \1\}$ following an analogy with quantum information theory. (Note $\0 \neq 0$!) It follows immediately from \cite{Me12_itsy_bitsy} that $SFH(\Sigma \times S^1, F \times S^1)$ is isomorphic to the $(n - \chi(\Sigma))$-th tensor power of this fundamental vector space $\Z_2 \0 \oplus \Z_2 \1$ (where $|F| = 2n$); that paper also describes in detail how to interpret the ``qubit'' elements of this vector space as sets of sutures. The isomorphism, without this interpretation, was shown earlier in \cite{HKM08}. Thus the above theorem also amounts to showing
\[
\widehat{HS}(D^2,F) \cong \left( \Z_2 \0 \oplus \Z_2 \1 \right)^{\otimes (n-1) }.
\]

Our notion of string homology applies more broadly than to surfaces with sutures. When we speak of a surface $\Sigma$ with signed points $F$ on the boundary, the points of $F$ only make sense for sutures if the signs of point of $F$ \emph{alternate} around each boundary component. We allow more general sets $F$, which we call \emph{markings}, as long as each boundary component has at least one point, and there are the same number of points of each sign. In this case string homology is well-defined, although sutures are not. Theorem \ref{thm:zero_thm} shows that both $\widehat{HS}$ and $HS^\infty$ are trivial in this case; and this is true not just for discs but for general $\Sigma$.

Although embedded contact homology plays a strong role in the motivations of this paper, none of the theorems directly assert an isomoprhism with $ECH$. This is partly because of the well-known difficulties in considering holomorphic curves near the boundary of the symplecitzation of a contact 3-manifold with boundary. It is also partly because the situation of $(\Sigma \times S^1, F \times S^1)$ is closer to the situation of \cite{Golovko09} than the situation of cotangent bundles considered by Cieliebak and Latschev in \cite{CL} and by the second author in \cite{Schoenfeld_Thesis}. Our string complex appears to require \emph{Reeb chords} as well as closed Reeb orbits in $ECH$ considerations. This is a matter for further investigation.

This paper is organised as follows. In section \ref{sec:string_complex} we define our basic concepts, including markings, sutures, string diagrams, and various forms of homotopy, importantly including spin homotopy. In section \ref{sec:the_string_complex} we define the string complexes $CS^\infty$ and $\widehat{CS}$. We consider their gradings, discuss bypasses and the $U$ map, and are then able to state our main theorems precisely. In section \ref{sec:properties_of_string_complexes} we prove various properties of the string complexes, prove that they are well-defined, and argue why our choices of definitions for $CS^\infty$ and $\widehat{CS}$ are apropriate. In section \ref{sec:non-alternating} we consider non-alternating markings and show that in this case homology is zero. In section \ref{sec:generalised_euler_class} we extend the notion of Euler class to general string diagrams on discs. In section \ref{sec:creation_annihilation} we define operators on the string complexes which are crucial for the proof. Then in section \ref{sec:homology_computation} we prove the theorem for $\widehat{HS}$, and in section \ref{sec:discs_with_spin} for $HS^\infty$.

\section{String diagrams}
\label{sec:string_complex}

\subsection{Markings, sutures and string diagrams}

Throughout, let $\Sigma$ be a compact oriented surface with nonempty boundary. 
\begin{defn}
A \emph{marking} $F$ on $\Sigma$ is a set of $2n$ points on $\partial \Sigma$, where $n \geq 1$, with $n$ points labelled ``in'' and the other $n$ points labelled ``out'', and at least one point on each component of $\partial \Sigma$. The pair $(\Sigma, F)$ is called a \emph{marked surface}. Write $F_{in}$ and $F_{out}$ for the corresponding points of $F$.
\end{defn}
Note different boundary components may have different (but always nonzero) numbers of points of $F$. Also, a boundary component may have a different number of ``in'' and ``out'' points.

\begin{defn}
A marking $F$ on $\Sigma$ is \emph{alternating} if, in order around each component of $\partial \Sigma$, the points of $F$ are labelled (in, out, \ldots, in, out).
\end{defn}
An alternating marked surface $(\Sigma,F)$ has the structure of a \emph{sutured background}, as in \cite{Me10_Sutured_TQFT, Me12_itsy_bitsy}. (Compare the terminology of \cite{Zarev09}.) As defined in \cite{Me12_itsy_bitsy}, this means that we can write $\partial \Sigma \backslash F = C_+ \sqcup C_-$, where $C_\pm$ is a collection of oriented arcs, $C_\pm$ is oriented as $\pm \partial \Sigma$, and $\partial C_\pm = -F$ as signed points; so the arcs of $C_+$ and $C_-$ alternate around $\partial \Sigma$. We will use both ``alternating marking'' and ``sutured background'' in this paper; the terms are synonymous.

An alternating $(\Sigma, F)$ is the boundary structure for a \emph{set of sutures} on $\Sigma$. Roughly speaking, a set of sutures on $(\Sigma, F)$ is a properly embedded set of curves $\Gamma$ with oriented boundary $F$ and cutting $\Sigma$ coherently into positive and negative regions $R_\pm$. Our definition follows \cite{Me12_itsy_bitsy}.
\begin{defn}
A \emph{set of sutures} $\Gamma$ on $(\Sigma, F)$ is a properly embedded oriented 1-submanifold of $\Sigma$ with $\partial \Gamma = F$, such that:
\begin{enumerate}
\item
$\Sigma \backslash \Gamma = R_+ \cup R_-$, where $R_\pm$ are surfaces oriented as $\pm \Sigma$;
\item
$\overline{\partial R_\pm \backslash \partial \Sigma} = \Gamma$ as oriented 1-manifolds; and
\item
for every component $C$ of $\partial \Sigma$, $C \cap \Gamma \neq \emptyset$.
\end{enumerate}
The pair $(\Sigma, \Gamma)$ is called a \emph{sutured surface}. 
\end{defn}
In particular, as we cross $\Gamma$ we proceed from $R_+$ to $R_-$ or vice versa. A component of $\Gamma$ is called a \emph{suture}. At each point of $\partial \Gamma = F$, precisely one suture either enters or exits $\Sigma$, according to the labelling on $F$. The arcs $C_\pm$ of the sutured background lie in the boundary of $R_\pm$; specifically, $\partial R_\pm = C_\pm \cup \Gamma$.

Thus, given a sutured background $(\Sigma, F)$ we may consider sets of sutures $\Gamma$ on $\Sigma$ such that $\partial \Gamma = F$; such $\Gamma$ ``fills in'' $(\Sigma, F)$.

As a generalisation of sutures, allowing curves to intersect and allowing a non-alternating marked surface as boundary data, we make the following definition.
\begin{defn}
A \emph{string diagram} $s$ on the marked surface $(\Sigma, F)$ is an immersed oriented 1-manifold in $\Sigma$, such that $\partial s = F$ (as signed points).
\end{defn}
That is, arcs of a string diagram run from $F_{in}$ to $F_{out}$. Generically a string diagram contains only transverse double intersections; this is general position. When a string diagram has no crossings and the complementary regions may be coherently oriented, it forms a set of sutures.

When needed, the string diagram $s$ can be given as an explicit immersion $s: \left( \sqcup_{i=1}^l S^1 \right) \sqcup \left( \sqcup_{i=1}^m [0,1] \right) \rightarrow (\Sigma, F)$, where there are $l$ arcs and $m$ closed curves in $s$. In practice we often abuse notation and identify this immersion with its image in $\Sigma$.

\subsection{Homotopy of string diagrams}

Several types of homotopy are useful for string diagrams.

Firstly, two string diagrams $s_0, s_1$ are \emph{homotopic} if there is a homotopy relative to endpoints from $s_0$ to $s_1$. Such a homotopy may introduce or remove intersections in the diagram, including self-intersections. We do not require that the homotopy be through immersions; thus the two string diagrams $s_0, s_1$ shown in figure \ref{fig:type_I_Reidemeister} are homotopic. In a homotopy from $s_t$ from $s_0$ to $s_1$ which changes the writhe of the string, a singularity will occur for some $s_t$.

\begin{figure}[ht]
\begin{center}
\begin{tikzpicture}[scale=1, 
string/.style={thick, draw=red, postaction={nomorepostaction, decorate, decoration={markings, mark=at position 0.5 with {\arrow{>}}}}}]

\draw [string] (0,-1) -- (0,1);
\draw [string] (2,-1) .. controls (2,0) and (2.5,.5) .. (2.5,0) .. controls (2.5,-0.5) and (2,0) .. (2,1);

\draw (1,0) node {$\longleftrightarrow$};
\draw (0,-1.5) node {$s_0$};
\draw (2,-1.5) node {$s_1$};

\end{tikzpicture}
\caption{Type I string Reidemeister move}
\label{fig:type_I_Reidemeister}
\end{center}
\end{figure}
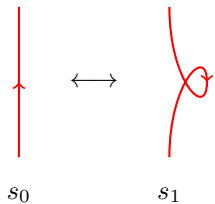

Secondly, two string diagrams $s_0, s_1$ are \emph{regular homotopic} if they are homotopic relative to endpoints \emph{through immersions}. The two string diagrams of figure \ref{fig:type_I_Reidemeister} are \emph{not} regular homotopic. A theorem of Whitney states that closed curves in the plane are regular homotopic if and only if they are homotopic and have the same winding number \cite{Whitney}.

A homotopy $s_t$ between two string diagrams $s_0, s_1$ is an \emph{ambient isotopy} if it arises from an isotopy of diffeomorphisms $F: \Sigma \times [0,1] \rightarrow \Sigma$ which hold $\partial \Sigma$ constant (i.e. each $F_t: \Sigma \times \{t\} \rightarrow \Sigma$ is a diffeomorphism, $F_0$ is the identity, and $s_t = F_t \circ s_0$, and for any $x \in \partial \Sigma$, $F_t(x)=x$). In an ambient isotopy of string diagrams, no crossings are altered; the strings move around the surface together. An ambient isotopy induces an isotopy rel boundary of the images of the immersions $s_0, s_1$, regarded as graphs on $\Sigma$.

Obviously every ambient isotopy of string diagrams is a regular homotopy, and every regular homotopy of string diagrams is a homotopy.

Any string diagram is homotopic to one in \emph{general position}, i.e. which has only transverse double intersection points. Just as for knot projections, two homotopic string diagrams in general position are related by a sequence of ambient isotopies and \emph{string Reidemeister moves}, as shown in figures \ref{fig:type_II_Reidemeister} and \ref{fig:type_III_Reidemeister}. Note that as string diagrams are oriented, there are two versions of the type II and III moves.

\begin{figure}[ht]
\begin{center}
\begin{tikzpicture}[scale=1, 
string/.style={thick, draw=red, postaction={nomorepostaction, decorate, decoration={markings, mark=at position 0.5 with {\arrow{>}}}}}]

\draw [string] (0,-1) -- (0,1);
\draw [string] (0.5,-1) -- (0.5,1);
\draw [string] (2.5,-1) .. controls (2.5,-0.5) and (3,-0.5) .. (3,0) .. controls (3,0.5) and (2.5,0.5) .. (2.5,1);
\draw [string] (3,-1) .. controls (3,-0.5) and (2.5,-0.5) .. (2.5,0) .. controls (2.5,0.5) and (3,0.5) .. (3,1);
\draw [string] (5,-1) -- (5,1);
\draw [string] (5.5,1) -- (5.5,-1);
\draw [string] (7.5,-1) .. controls (7.5,-0.5) and (8,-0.5) .. (8,0) .. controls (8,0.5) and (7.5,0.5) .. (7.5,1);
\draw [string] (8,1) .. controls (8,0.5) and (7.5,0.5) .. (7.5,0) .. controls (7.5,-0.5) and (8,-0.5) .. (8,-1);
\draw (1.5,0) node {$\longleftrightarrow$};
\draw (6.5,0) node {$\longleftrightarrow$};

\end{tikzpicture}
\caption{Type II string Reidemeister moves.}
\label{fig:type_II_Reidemeister}
\end{center}
\end{figure}
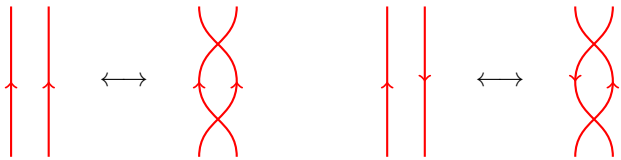

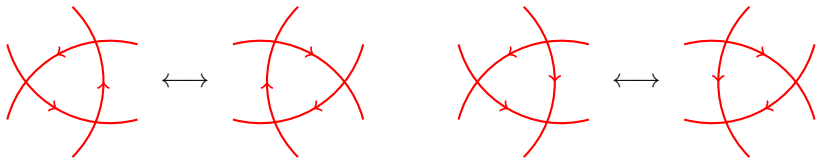
\begin{figure}[ht]
\begin{center}
\begin{tikzpicture}[scale=1, 
string/.style={thick, draw=red, postaction={nomorepostaction, decorate, decoration={markings, mark=at position 0.5 with {\arrow{>}}}}}]

\draw [string] ($ (-6,0) + (-90:1) $) to [bend right=45] ($ (-6,0) + (90:1) $);
\draw [string] ($ (-6,0) + (30:1) $) to [bend right=45] ($ (-6,0) + (210:1) $);
\draw [string] ($ (-6,0) + (150:1) $) to [bend right=45] ($ (-6,0) + (-30:1) $);

\draw (-4.5,0) node {$\longleftrightarrow$};

\draw [string] ($ (-3,0) + (-90:1) $) to [bend left=45] ($ (-3,0) + (90:1) $);
\draw [string] ($ (-3,0) + (30:1) $) to [bend left=45] ($ (-3,0) + (210:1) $);
\draw [string] ($ (-3,0) + (150:1) $) to [bend left=45] ($ (-3,0) + (-30:1) $);

\draw [string] ($ (0,0) + (90:1) $) to [bend left=45] ($ (0,0) + (-90:1) $);
\draw [string] ($ (0,0) + (30:1) $) to [bend right=45] ($ (0,0) + (210:1) $);
\draw [string] ($ (0,0) + (150:1) $) to [bend right=45] ($ (0,0) + (-30:1) $);

\draw (1.5,0) node {$\longleftrightarrow$};

\draw [string] ($ (3,0) + (90:1) $) to [bend right=45] ($ (3,0) + (-90:1) $);
\draw [string] ($ (3,0) + (30:1) $) to [bend left=45] ($ (3,0) + (210:1) $);
\draw [string] ($ (3,0) + (150:1) $) to [bend left=45] ($ (3,0) + (-30:1) $);

\end{tikzpicture}
\caption{Type III string Reidemeister moves.}
\label{fig:type_III_Reidemeister}
\end{center}
\end{figure}

Similarly, any string diagram is regular homotopic to one in general position, and two regular homotopic string diagrams are related by a sequence of ambient isotopies and string Reidemeister moves of type II and III (not type I, which changes winding number and regular homotopy class).

For example, when $\Sigma = D^2$, any string diagram without closed curves is homotopic to one consisting of straight line chords between points of $F$. Any string diagram without closed curves is \emph{regular} homotopic to one consisting of chords, each with a fixed number of ``whirls'' giving the correct winding number. The ambient isotopy classes of string diagrams on $D^2$ are much more complicated, since in general string diagrams can have curves intersecting obnoxiously.

\subsection{Spin homotopy}
\label{sec:spin_homotopy}

For our purposes it will be useful to define a notion of ``spin homotopy'', which is closely related to regular homotopy. Roughly, a spin homotopy of string diagrams is a regular homotopy, but it also allows type I Reidemeister moves, putting ``whirls'' in the strings, and altering the winding number of a string, as long as the total change in winding number is zero.

First, consider the operation, a type I Reidemeister move, of taking an embedded arc which forms part of a string diagram, and putting a ``whirl'' there (see figure \ref{fig:type_I_Reidemeister}). Since the strings of a string diagram are oriented, that whirl may run clockwise or anticlockwise, changing the winding number of the string by $-1$ or $+1$ respectively. We call that type I Reidemeister move \emph{negative} or \emph{positive} accordingly.

\begin{defn}
A \emph{balanced} type I Reidemeister move on a string diagram $s$ consists of taking two disjoint embedded arcs $a_-, a_+$, in $s$, performing a negative type I Reidemeister move on $a_-$, and a positive type I Reidemeister move on $a_+$.
\end{defn}
(Note the two arcs $a_-, a_+$ may lie on the same immersed curve of $s$, or not.)

\begin{defn}
A \emph{spin homotopy} of string diagrams is a homotopy which can be expressed as a sequence of the following:
\begin{enumerate}
\item
ambient isotopies;
\item
type II string Reidemeister moves;
\item
type III string Reidemeister moves;
\item
balanced type I Reidemeister moves.
\end{enumerate}
\end{defn}

Thus, a regular homotopy, which never uses type I Reidemeister moves, is a spin homotopy. Similarly, a spin homotopy is a homotopy.
\[
\text{ambient isotopy} \subset \text{regular homotopy} \subset \text{spin homotopy} \subset \text{homotopy}
\]

Given two homotopic string diagrams $s_0, s_1$, there exists a unique integer $n$ such that introducing $n$ whirls into $s_0$ (at any possible locations) to obtain a string diagram $s'_0$, the string diagrams $s'_0, s_1$ are spin homotopic. We call this integer $n$ the \emph{relative winding of $s_1$ with respect to $s_0$}. Spin homotopic string diagrams have relative winding of $0$.

When $\Sigma$ is a disc, we can in fact replace the notion of ``relative'' with ``absolute'' winding. In section \ref{sec:generalised_euler_class} we will assign an integer $e(s)$ to a general string diagram $s$ on a disc, such that $e(s_1) - e(s_0)$ is the relative winding of $s_1$ with respect to $s_0$. In the case that $s$ is a set of sutures $\Gamma$, we will show that this $e(\Gamma)$ is the \emph{Euler class} of $s$, which is defined as $e(\Gamma) = \chi(R_+) - \chi(R_-)$. We will also call $e(s)$ the Euler class of the string diagram. So, on $D^2$, a homotopy class of string diagrams splits into a countable infinity of spin homotopy classes, which are indexed precisely by the Euler class.

\section{The string complex}
\label{sec:the_string_complex}

\subsection{Definition of the complex}

\begin{defn}
Given a marked surface $(\Sigma, F)$, we define the following vector spaces over $\Z_2$:
\begin{enumerate}
\item
$CS^\infty(\Sigma,F)$ is freely generated by spin homotopy classes of string diagrams on $(\Sigma,F)$.
\item
$\widehat{CS}(\Sigma,F)$ is freely generated by homotopy classes of string diagrams on $(\Sigma,F)$ which contain no contractible closed curves.
\end{enumerate}
\end{defn}

Since a spin homotopy is a homotopy, there is a natural map $p : CS^\infty(\Sigma,F) \rightarrow \widehat{CS}(\Sigma,F)$ which is the identity on string diagrams without contractible closed curves, and sends string diagrams with contractible closed curves to $0$.

For example, when $\Sigma = D^2$ and $|F| = 2n$, $\dim \widehat{CS}(\Sigma,F) = n!$, and $CS^\infty(\Sigma,F)$ is infinite dimensional.

\subsection{Grading by intersections}

Both $CS^\infty$ and $\widehat{CS}$ are naturally graded according to number of intersections of the curves in a string diagram. A string diagram $s$ in general position has a finite number of crossings. If we consider the homotopy class of $s$ then there is a string diagram in that class which has a minimal number of crossings, which we denote $\widehat{I}(s)$. Similarly, if we consider the spin homotopy class of $s$, there is a minimal number of crossings, which we denote $I^\infty(s)$.

A nonzero element $v \in CS^\infty(\Sigma,F)$ can be written as a finite sum $v = \sum_j s_j$, where $s_j$ is a string diagram up to spin homotopy. We define $I^\infty (v) = \max_j I^\infty (s_j)$, and let $CS^\infty_i (\Sigma, F)$ be the free $\Z_2$-vector space generated by (spin homotopy classes of) diagrams $s$ with $I^\infty(s) = i$. Similarly, an element $v \in \widehat{CS}(\Sigma,F)$ can be written as $\sum_j s_j$ where $s_j$ are string diagrams without contractible loops up to homotopy. We let $\widehat{I}(v) = \max_j \widehat{I}(s_j)$, and let $\widehat{CS}_i (\Sigma, F)$ be spanned by (homotopy classes of) diagrams with $\widehat{I}(s) = i$. We also set $I^\infty(0) = \widehat{I}(0) = -\infty$ for completeness. We then have
\[
CS^\infty (\Sigma, F) = \bigoplus_{i \geq 0} CS^\infty_i (\Sigma, F), \quad
\widehat{CS}(\Sigma, F) = \bigoplus_{i \geq 0} \widehat{CS}_i (\Sigma, F), \quad
\]
and the map $p: CS^\infty (\Sigma,F) \rightarrow \widehat{CS}(\Sigma,F)$ decreases grading: $\widehat{I} (p(v)) \leq I^\infty (v)$.

Among the string diagrams with $0$ crossings are those string diagrams which are sets of sutures, i.e. when the curves cut $(\Sigma, F)$ into coherently oriented regions. A string diagram $s$ may be spin homotopic to a set of sutures; if so, that set of sutures is unique. We write $CS^\infty_{sut} (\Sigma,F)$ for the subspace of $CS^\infty_0(\Sigma,F)$ generated by spin homotopy classes of sets of sutures. Isotopy classes of sutures form a basis for $CS^\infty_{sut}(\Sigma,F)$.

Similarly, a string diagram $s$ may be homotopic to a set of sutures, and if so that set of sutures is unique. We define $\widehat{CS}_{sut} (\Sigma, F)$ for the subspace of $\widehat{CS}_0 (\Sigma, F)$ generated by homotopy classes of sutures without contractible closed curves. The isotopy classes of sutures without contractible loops form a basis for $\widehat{CS}_{sut}(\Sigma,F)$.

\subsection{Grading by Euler class}

As noted above, when $\Sigma = D^2$, we will define an Euler class $e(s)$ of a string diagram $s$, which is constant on spin homotopy classes (but not on homotopy classes in general). This gives another grading on $CS^\infty(D^2,F)$ (but not on $\widehat{CS}(D^2,F)$). We write $CS^\infty_e(D^2,F)$ for the span of (spin homotopy classes of) string diagrams of Euler class $e$, and $CS^\infty_{e,i} (D^2,F)$ for the span of (spin homotopy classes of) string diagrams $s$ with Euler class $e$ and $I^\infty (s) = i$. Then
\[
CS^\infty(D^2,F) = \bigoplus_{e \in \Z} CS^\infty_e (D^2, F) = \bigoplus_{e \in \Z} \bigoplus_{i \geq 0} CS^\infty_{e,i} (D^2,F),
\quad
CS^\infty_i (D^2,F) = \bigoplus_{e \in \Z} CS^\infty_{e,i} (D^2, F)
\]
Restricting to sets of sutures, the Euler class gives a grading on $CS^\infty_{sut}(D^2,F)$. Writing $CS^\infty_{sut,e} (D^2, F)$ for the span of sutures of Euler class $e$ we have $CS^\infty_{sut}(D^2, F) = \oplus_{e \in \Z} CS^\infty_{sut,e} (D^2, F)$.

As we will define it, the Euler class is not well defined on homotopy classes of sutures, hence not on $\widehat{CS}(D^2,F)$. But the Euler class is well-defined on sutures; and if a string diagram is homotopic to a set of sutures, then the set of sutures is unique. So we obtain a grading $\widehat{CS}_{sut}(D^2,F) = \oplus_{e \in \Z} \widehat{CS}_{sut,e} (D^2,F)$.

\subsection{Bypass triples of sutures}

There are distinguished triples of sets of sutures on $(\Sigma,F)$ known as \emph{bypass triples}. In \cite{Hon00I}, Honda introduced a contact-geometric operation known as \emph{bypass addition}, which has the effect of performing an operation on a dividing set on a convex surface. Dividing sets can be regarded as sutures, and the operation on sutures is called \emph{bypass surgery}.

Bypass surgery consists of taking an embedded disc $D$ in a sutured surface $(\Sigma,  \Gamma)$, such that $D \cap \Gamma$ consists of $3$ disjoint parallel arcs, and replacing the sutures by ``60 degree rotation'' as shown in figure \ref{fig:bypass_relation}. This ``rotation'' of sutures can be done in two possible ways, known as \emph{upwards} or \emph{downwards} bypass surgery. Bypass surgery on a string diagram may produce more or less closed sutures, but the result is always a set of sutures. Any two sets of sutures related by bypass surgery determine a third set of sutures related to them both.

Bypass surgery is an order $3$ operation, and sets of sutures related by bypass surgery along the same $D$ come in triples, called \emph{bypass triples}. A bypass triple of sutures $\Gamma_1, \Gamma_2, \Gamma_3$ on $(\Sigma, F)$ can be regarded as an element $\Gamma_1 + \Gamma_2 + \Gamma_3$ of $CS^\infty_{sut} (\Sigma, F)$ or of $\widehat{CS}_{sut}(\Sigma,F)$ (setting sutures with closed loops equal to zero, using the map $p: CS^\infty(\Sigma,F) \To \widehat{CS}(\Sigma,F)$, which takes $CS^\infty_{sut}(\Sigma,F)$ to $\widehat{CS}_{sut}(\Sigma,F)$).
\begin{defn}\
\begin{enumerate}
\item
The $\Z_2$-vector space $\Byp^\infty(\Sigma,F)$ is the subspace of $CS^\infty_{sut}(\Sigma,F)$ spanned by bypass triples.
\item
The $\Z_2$-vector space $\widehat{\Byp}(\Sigma,F)$ is the subspace of $\widehat{CS}_{sut}(\Sigma,F)$ spanned by bypass triples.
\end{enumerate}
\end{defn}

Obviously $p(\Byp^\infty(\Sigma,F)) = \widehat{\Byp}(\Sigma,F)$. 

We note that bypass triples are defined not just for sutures, but for string diagrams in general: there is a more general notion of bypass surgery.

\subsection{Resolving crossings and differential}

Since the curves of a string diagram $s$ are oriented, any transverse double crossing $x$ of $s$ has a natural resolution; see figure \ref{fig:resolving_crossing}. After this resolution we have a string diagram, well defined up to ambient isotopy, with one fewer crossings, which we denote $r_x(s)$. This resolution may add or remove curves to or from $s$.

The idea is to set, for a string diagram $s$:
\[
\partial(s) = \sum_{x \text{ crossing of } s} r_x(s).
\]
This is a formal sum of string diagrams. Each $r_x(s)$ is well-defined up to ambient isotopy, hence up to spin homotopy and up to homotopy of string diagrams. We will prove (lemma \ref{lem:homotopic_diff}) that $\partial$ is actually well-defined on \emph{homotopy classes} of string diagrams without contractible loops, and on \emph{spin homotopy} classes of string diagrams in general. Hence we will obtain a well-defined linear map $\partial$ on both $CS^\infty(\Sigma,F)$ and $\widehat{CS}(\Sigma,F)$, which decreases the intersection gradings $I^\infty$, $\widehat{I}$.

Consider $\partial^2 (s)$; this is the sum of string diagrams obtained by resolving ordered pairs of double points. Obviously the diagram obtained by resolving crossing $x$ then crossing $y$ is ambient isotopic to the diagram obtained by resolving $y$ then $x$; so, once we have proved $\partial$ is well-defined, it's clear $\partial^2 = 0$ on both $CS^\infty(\Sigma,F)$ and $\widehat{CS}(\Sigma,F)$. It will then follow that the homologies $HS^\infty (\Sigma,F) = H( CS^\infty(\Sigma,F), \partial)$ and $\widehat{HS}(\Sigma,F) = H ( \widehat{CS}(\Sigma,F), \partial)$ are well defined.

We will also (section \ref{sec:why_these}) show why $\partial$ does \emph{not} define a differential on other similar vector spaces of string diagrams; for instance, not on regular homotopy classes of string diagrams. This motivates our particular chain complexes.

Once we have defined the Euler class of a string diagram on $D^2$, it will not be difficult to show that the differential preserves the Euler class, so that
\[
HS^\infty(D^2,F) = \bigoplus_e HS^\infty_e (D^2,F) \quad \text{ where } \quad HS^\infty_e (D^2,F) = H \left( CS^\infty_e (D^2, F), \partial \right).
\]

\subsection{The U map}
\label{sec:U_map}

In the $HF^\infty$ version of Heegaard Floer theory, there is a $U$ map; and likewise in $ECH$. This map has the effect of changing grading and counts some type of intersection. The resulting algebraic objects essentially become $\Z[U,U^{-1}]$-modules. Something roughly analogous happens with $HS^\infty$ and we will name the map obtained $U$.

Consider the spin homotopy class of a string diagram $s$ on a marked surface $(\Sigma,F)$. By definition, within this class we can perform ambient isotopies, type II and III string Reidemeister moves, and balanced type I string Reidemeister moves. The $U$ map simply performs ``unbalanced'' type I string Reidemeister moves, adding two anticlockwise whirls. Since $s$ is only defined up to spin homotopy, this anticlockwise whirl may be added anywhere in the diagram, and the result is well-defined up to spin homotopy. Similarly, the $U^{-1}$ map adds two clockwise whirls. It's not difficult to see that applying $U$ and then $U^{-1}$ to $s$ results in a string diagram spin homotopic to $s$. On $D^2$, we will see that $U^n$ adjusts the Euler class of $s$ by $4n$, i.e. $e(U^n s) = e(s) + 4n$.

It may seem somewhat curious that the $U$ map is given by adding \emph{two} whirls. However we will see in section \ref{sec:spin_base_case} that adding a single whirl gives a string diagram that is zero in homology.

In any case, there is a $\Z_2[U,U^{-1}]$ action on $CS^\infty(\Sigma,F)$, and it becomes a $\Z_2[U,U^{-1}]$-module. We will show that this in fact descends to an action on homology, so that $HS^\infty(\Sigma,F)$ is a $\Z_2[U,U^{-1}]$-module. Also $CS_{sut}^\infty(\Sigma,F)$ and $\Byp^\infty(\Sigma,F)$ can be regarded as $\Z_2[U,U^{-1}]$-modules, respectively generated by sutures and bypass triples. Note however that applying $U$ to a set of sutures results in a string diagram that is no longer a set of sutures; when we consider these spaces it will always be over $\Z_2[U,U^{-1}]$ and so this larger class of diagrams will be considered.

\subsection{Statements of main theorems}

Our main theorems are descriptions of the above homologies. They clearly include the statements in the introduction.

Theorem \ref{thm:non_alternating_homology} states that $HS^\infty(\Sigma,F)= \widehat{HS}(\Sigma,F) = 0$ when $F$ is not alternating; this is now a precise statement. We can state theorems \ref{thm:rough_main_thm1} and \ref{thm:rough_main_thm2} in full generality.

\begin{thm}
\label{thm:main_thm}
\label{thm:bypass_relation}
Let $(D^2,F)$ be an alternating marked disc with $|F|=2n$. Then
\[
\widehat{HS}(D^2,F) \cong SFH(D^2 \times S^1, F \times S^1) \cong \left( \Z_2 \0 \oplus \Z_2 \1 \right)^{\otimes (n - 1)},
\]
where the isomorphism $\widehat{HS}(D^2,F) \cong SFH(D^2 \times S^1, F \times S^1)$ is induced by the map which sends a set of sutures $\Gamma$ to the contact element of the corresponding contact structure on $(D^2 \times S^1, F \times S^1)$.

Moreover:
\begin{enumerate}
\item
Any nonzero homology class in $\widehat{HS}(D^2,F)$ is represented by a linear combination of string diagrams which are sets of sutures on $(D^2,F)$; in other words, the map
\[
\widehat{CS}_{sut}(D^2, F) \hookrightarrow \ker \partial \To \widehat{HS}(D^2,F)
\]
induced by inclusions and quotient by boundaries is surjective. The kernel of this map is precisely the span of bypass triples, hence
\[
\widehat{HS}(D^2,F) \cong \frac{\widehat{CS}_{sut}(D^2,F)}{\widehat{\Byp}(D^2,F)}.
\]
\item
The above isomorphisms restrict to Euler graded summands
\[
\widehat{HS}_e(D^2,F) \cong SFH_e(D^2 \times S^1, F \times S^1) \cong \bigoplus_{\substack{e_i \in \{\0, \1\} \\ \# \1 - \# \0 = e}} e_1 \otimes \cdots \otimes e_n
\]
and
\[
\widehat{HS}_e(D^2,F) \cong \frac{\widehat{CS}_{sut,e}(D^2,F)}{\widehat{\Byp}_e(D^2,F)}.
\]
\end{enumerate}
\end{thm}

The description of $HS^\infty$ is similar, but with the $U$ map giving extra structure. In essence, we just take the previous answer and allow everything to be multiplied by powers of $U$; this amounts to tensoring with $\Z_2[U,U^{-1}]$.
\begin{thm}
\label{thm:main_theorem_infty}
Let $(D^2, F)$ be an alternating marked disc with $|F| = 2n$. Then
\[
HS^\infty(D^2,F) = \Z_2 [U,U^{-1}] \otimes \widehat{HS}(D^2,F).
\]
In particular:
\begin{enumerate}
\item
Any nonzero homology class in $HS^\infty(D^2,F)$ is represented by a $\Z_2[U,U^{-1}]$-linear combination of string diagrams which are sets of sutures on $(D^2,F)$. We have, as $\Z_2[U,U^{-1}]$-modules:
\[
HS^\infty(D^2,F) \cong \frac{CS^\infty_{sut}(D^2,F)}{\Byp^\infty(D^2,F)}
\cong \Z_2[U,U^{-1}] \otimes \left( \frac{\widehat{CS}_{sut}(D^2,F)}{\widehat{\Byp}(D^2,F)} \right).
\]
\item
Over $\Z_2$, $HS^\infty(D^2,F)$ decomposes over powers of $U$, and over Euler class, as
\[
HS^\infty(D^2,F) \cong \bigoplus_{j \in \Z} U^j \widehat{HS}(D^2,F) \cong \bigoplus_{j \in \Z} \bigoplus_{e \in \Z} U^j \widehat{HS}_e(D^2,F),
\]
and
\[
HS^\infty_e(D^2,F) \cong \bigoplus_{j \in \Z} U^j \widehat{HS}_{e-4j}(D^2,F)
\cong \bigoplus_{j \in \Z} U^j \frac{ \widehat{CS}_{sut,e-4j}(D^2,F)}{\widehat{\Byp}_{e-4j}(D^2,F)}.
\]
\end{enumerate}
\end{thm}

At this stage we may note the following:
\begin{enumerate}
\item
In $HS^\infty(D^2,F)$, any string diagram $s$ with a contractible loop is zero. \item
The decompositions in (ii) are pure algebraic manipulations, since $\Z_2[U,U^{-1}] = \oplus_{j \in \Z} U^j \Z_2$. We also use the fact that $U$ raises Euler class by $4$ in the second set of decompositions.
\end{enumerate}

In essence, in both variants of $HS$: all homology lies in sutures, and the only relation between these sutures is the bypass relation.

In \cite{Me09Paper} the first author defined a vector space $SFH_{comb}(T,n)$ to be the $\Z_2$ vector space generated by chord diagrams on the disc, which are sutures without closed curves (i.e. $\widehat{CS}_{sut}(D^2,F)$), modulo the bypass relation, i.e.
\[
SFH_{comb}(T,n) = \frac{\widehat{CS}_{sut}(D^2, F)}{\widehat{\Byp}(D^2,F)}.
\]
The first author gave a natural basis of chord diagrams / sutures for this space, and a natural partial order on this basis; described general chord diagrams with respect to this basis; related the various spaces $SFH_{comb}(T,n)$ via various operators; and considered relations to contact geometry, category theory, and sutured Floer homology. In particular, $SFH_{comb}(T,n) \cong SFH(T,n)$, the ($\Z_2$) sutured Floer homology of the solid torus $T$ with $n$ pairs of longitudinal sutures. In \cite{Me10_Sutured_TQFT} these considerations were extended to $\Z$ coefficients, and in \cite{Me12_itsy_bitsy} to general surfaces $(\Sigma,F)$. The above isomorphisms can be regarded as another type of combinatorial description of sutured Floer homology.

\section{Properties of the string complexes}
\label{sec:properties_of_string_complexes}

\subsection{Well-definition}

We first show that the differential $\partial$ makes $CS^\infty (\Sigma, F)$ and $\widehat{CS}(\Sigma,F)$ into well-defined chain complexes. We can always assume, after performing a regular homotopy if necessary, that a string diagram is in general position.

First consider $CS^\infty(\Sigma,F)$. An element $v \in CS^\infty(\Sigma,F)$ is given as $v = \sum_{j=1}^m s_j$, where $s_j$ are distinct string diagrams, up to spin homotopy. Two such elements $v = \sum_{j=1}^m s_j$, $v' = \sum_{j=1}^{m'} s'_j$ are equal in $CS^\infty(\Sigma,F)$ if and only if there is a bijection between the $\{s_j\}_{j=1}^m$ and $\{s'_j\}_{j=1}^{m'}$ with corresponding $s_j$ and $s'_j$ spin homotopic; in this case we say $v,v'$ are \emph{spin homotopic}. More generally, we will call formal sums of diagrams \emph{homotopic}, or \emph{regular homotopic}, or \emph{ambient isotopic} when their terms are bijective and the corresponding diagrams are homotopic of the corresponding type.

\begin{lem}
\label{lem:homotopic_diff}
The map $\partial : CS^\infty(\Sigma, F) \rightarrow CS^\infty(\Sigma, F)$ is well defined. That is, if string diagrams $s,s'$ are spin homotopic then $\partial s, \partial s'$ are spin homotopic.
\end{lem}

\begin{proof}
Clearly if $s,s'$ are ambient isotopic string diagrams then $\partial s, \partial s'$ are sums of ambient isotopic string diagrams, hence ambient isotopic. It remains then to show that if $s,s'$ are related by a string Reidemeister II, III or balanced Reidemeister I move then $\partial s, \partial s'$ are spin homotopic.

Figures \ref{fig:type_I_differential}, \ref{fig:type_II_differential} and \ref{fig:type_III_differential} show that this is the case. In each we show the local effect of the Reidemeister moves. In applying $\partial$, we must resolve differentials both in the region where the Reidemeister move is performed, and also outside that region. For the balanced Reidemeister I move we have two local regions where Reidemeister moves are performed; there are two diagrams obtained after two whirls are added, and these are spin homotopic. There are two types of Reidemeister II and III moves (up to symmetry), since the strands are oriented, which affects how crossings are resolved. Performing $\partial$ on two diagrams related by a type II move gives diagrams which are spin homotopic, though not always regular homotopic; on two diagrams related by a type III move, gives diagrams which are regular homotopic. We conclude $\partial s, \partial s'$ are spin homotopic, and hence $\partial s = \partial s'$ in $CS(\Sigma,F)$.

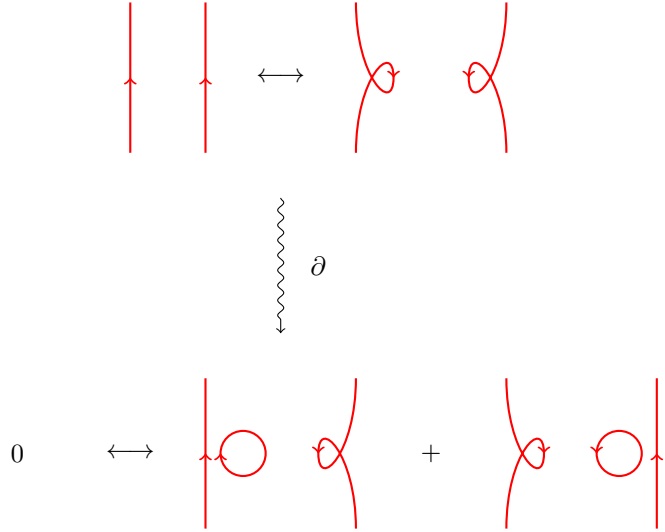
\begin{figure}[ht]
\begin{center}
\begin{tikzpicture}[scale=1, 
string/.style={thick, draw=red, postaction={nomorepostaction, decorate, decoration={markings, mark=at position 0.5 with {\arrow{>}}}}}]

\draw [string] (-1,-1) -- (-1,1);
\draw [string] (0,-1) -- (0,1);
\draw (1,0) node {$\longleftrightarrow$};
\draw [string] (2,-1) .. controls (2,0) and (2.5,.5) .. (2.5,0) .. controls (2.5,-0.5) and (2,0) .. (2,1);
\draw [string] (4,-1) .. controls (4,0) and (3.5,.5) .. (3.5,0) .. controls (3.5,-0.5) and (4,0) .. (4,1);

\draw [shorten >=1mm, -to, decorate, decoration={snake,amplitude=.4mm, segment length = 2mm, pre=moveto, pre length = 1mm, post length = 2mm}]
(1,-1.5) -- (1,-3.5);
\draw (1.5,-2.5) node {$\partial$};

\draw (-2.5,-5) node {$0$};
\draw (-1,-5) node {$\longleftrightarrow$};
\draw [string] (0,-6) -- (0,-4);
\draw [string] (0.8,-5) arc(0:-360:0.3);
\draw [string] (2,-6) .. controls (2,-5) and (1.5,-4.5) .. (1.5,-5) .. controls (1.5,-5.5) and (2,-5) .. (2,-4);
\draw (3,-5) node {$+$};
\draw [string] (4,-6) .. controls (4,-5) and (4.5,-4.5) .. (4.5,-5) .. controls (4.5,-5.5) and (4,-5) .. (4,-4);
\draw [string] (5.5,-5) circle (0.3 cm);
\draw [string] (6,-6) -- (6,-4);

\end{tikzpicture}
\caption{Balanced type I string Reidemeister move and differential.}
\label{fig:type_I_differential}
\end{center}
\end{figure}

\begin{figure}[ht]
\begin{center}
\begin{tikzpicture}[scale=1, 
string/.style={thick, draw=red, postaction={nomorepostaction, decorate, decoration={markings, mark=at position 0.5 with {\arrow{>}}}}}]

\draw [string] (0,-1) -- (0,1);
\draw [string] (0.5,-1) -- (0.5,1);
\draw (1.5,0) node {$\longleftrightarrow$};
\draw [string] (2.5,-1) .. controls (2.5,-0.5) and (3,-0.5) .. (3,0) .. controls (3,0.5) and (2.5,0.5) .. (2.5,1);
\draw [string] (3,-1) .. controls (3,-0.5) and (2.5,-0.5) .. (2.5,0) .. controls (2.5,0.5) and (3,0.5) .. (3,1);
\draw [string] (5,-1) -- (5,1);
\draw [string] (5.5,1) -- (5.5,-1);
\draw [string] (7.5,-1) .. controls (7.5,-0.5) and (8,-0.5) .. (8,0) .. controls (8,0.5) and (7.5,0.5) .. (7.5,1);
\draw [string] (8,1) .. controls (8,0.5) and (7.5,0.5) .. (7.5,0) .. controls (7.5,-0.5) and (8,-0.5) .. (8,-1);
\draw (6.5,0) node {$\longleftrightarrow$};

\draw [shorten >=1mm, -to, decorate, decoration={snake,amplitude=.4mm, segment length = 2mm, pre=moveto, pre length = 1mm, post length = 2mm}]
(4,-1.5) -- (4,-3.5);
\draw (4.5,-2.5) node {$\partial$};

\draw (-1.25,-5) node {$0$};
\draw (0,-5) node {$\longleftrightarrow$};
\draw [string] (1,-6) -- (1,-5)  .. controls (1,-4.5) and (1.5,-4.5) .. (1.5,-4);
\draw [string] (1.5,-6) -- (1.5,-5) .. controls (1.5,-4.5) and (1,-4.5) .. (1,-4);
\draw (2,-5) node {$+$};
\draw [string] (2.5,-6) .. controls (2.5,-5.5) and (3,-5.5) .. (3,-5) -- (3,-4);
\draw [string] (3,-6) .. controls (3,-5.5) and (2.5,-5.5) .. (2.5,-5) -- (2.5,-4);
\draw (5.25,-5) node {$0$};
\draw (6.5,-5) node {$\longleftrightarrow$};
\draw [string] (8,-4) .. controls (8,-4.5) and (7.5,-4.5) .. (7.5,-5) .. controls (7.5,-5.5) and (8,-5.5) .. (8,-5) .. controls (8,-4.5) and (7.5,-4.5) .. (7.5,-4);
\draw [string] (7.5,-6) .. controls (7.5,-5.5) and (8,-5.5) .. (8,-6);
\draw (8.5,-5) node {$+$};
\draw [string] (9,-6) .. controls (9,-5.5) and (9.5,-5.5) .. (9.5,-5) .. controls (9.5,-4.5) and (9,-4.5) .. (9,-5) .. controls (9,-5.5) and (9.5,-5.5) .. (9.5,-6);
\draw [string] (9.5,-4) .. controls (9.5,-4.5) and (9,-4.5) .. (9,-4);

\end{tikzpicture}
\caption{Type II string Reidemeister moves and differential.}
\label{fig:type_II_differential}
\end{center}
\end{figure}

\begin{figure}[ht]
\begin{center}
\begin{tikzpicture}[scale=0.8, 
string/.style={thick, draw=red, postaction={nomorepostaction, decorate, decoration={markings, mark=at position 0.5 with {\arrow{>}}}}}]

\draw [xshift=-1.5 cm, string] (-90:1) to [bend right=45] (90:1);
\draw [xshift=-1.5 cm, string] (30:1) to [bend right=45] (210:1);
\draw [xshift=-1.5 cm, string] (150:1) to [bend right=45] (-30:1);

\draw (0,0) node {$\longleftrightarrow$};

\draw [xshift=1.5 cm, string] (-90:1) to [bend left=45] (90:1);
\draw [xshift=1.5 cm, string] (30:1) to [bend left=45] (210:1);
\draw [xshift=1.5 cm, string] (150:1) to [bend left=45] (-30:1);

\draw [shorten >=1mm, -to, decorate, decoration={snake,amplitude=.4mm, segment length = 2mm, pre=moveto, pre length = 1mm, post length = 2mm}]
(0,-1.5) -- (0,-2.5);
\draw (0.5,-2) node {$\partial$};

\draw [xshift=-7.5 cm, yshift=-4 cm, string] (-90:1) to [bend right=45] (90:1);
\draw [xshift=-7.5 cm, yshift=-4 cm, string] (30:1) .. controls ($ (30:1) + (180:0.5) $) and ($ (180:0.3) + (90:0.5) $) .. (180:0.3) .. controls ($ (180:0.3) + (-90:0.5) $) and ($ (-30:1) + (180:0.5) $) .. (-30:1);
\draw [xshift=-7.5 cm, yshift=-4 cm, string] (150:1) to [bend left=90] (210:1);

\draw (-6,-4) node {$+$};

\draw [xshift=-4.5 cm, yshift=-4 cm, string] (-90:1) to [bend left=90] (-30:1);
\draw [xshift=-4.5 cm, yshift=-4 cm, string] (30:1) to [bend right=45] (210:1);
\draw [xshift=-4.5 cm, yshift=-4 cm, string] (150:1) .. controls ($ (150:1) + (-60:0.5) $) and ($ (-60:0.3) + (-150:0.5) $) .. (-60:0.3) .. controls ($ (-60:0.3) + (30:0.5) $) and ($ (90:1) + (-60:0.5) $) .. (90:1);

\draw (-3,-4) node {$+$};

\draw [xshift=-1.5 cm, yshift=-4 cm, string] (-90:1) .. controls ($ (-90:1) + (60:0.5) $) and ($ (60:0.3) + (-30:0.5) $) .. (60:0.3) .. controls ($ (60:0.3) + (150:0.5) $) and ($ (-150:1) + (60:0.5) $) .. (-150:1);
\draw [xshift=-1.5 cm, yshift=-4 cm, string] (30:1) to [bend left=90] (90:1);
\draw [xshift=-1.5 cm, yshift=-4 cm, string] (150:1) to [bend right=45] (-30:1);
\draw (0,-4) node {$\longleftrightarrow$};

\draw [xshift=1.5 cm, yshift=-4 cm, string] (30:1) to [bend left=45] (210:1);
\draw [xshift=1.5 cm, yshift=-4 cm, string] (150:1) to [bend right=90] (90:1);
\draw [xshift=1.5 cm, yshift=-4 cm, string] (-90:1) .. controls ($ (-90:1) + (120:0.5) $) and ($ (120:0.3) + (-150:0.5) $) .. (120:0.3) .. controls ($ (120:0.3) + (30:0.5) $) and ($ (-30:1) + (120:0.5) $) .. (-30:1);

\draw (3,-4) node {$+$};

\draw [xshift=4.5 cm, yshift=-4 cm, string] (30:1) .. controls ($ (30:1) + (-120:0.5) $) and ($ (-120:0.3) + (-30:0.5) $) .. (-120:0.3) .. controls ($ (-120:0.3) + (150:0.5) $) and ($ (90:1) + (-120:0.5) $) .. (90:1);
\draw [xshift=4.5 cm, yshift=-4 cm, string] (150:1) to [bend left=45] (-30:1);
\draw [xshift=4.5 cm, yshift=-4 cm, string] (-90:1) to [bend right=90] (-150:1);

\draw (6,-4) node {$+$};

\draw [xshift=7.5 cm, yshift=-4 cm, string] (30:1) to [bend right=90] (-30:1);
\draw [xshift=7.5 cm, yshift=-4 cm, string] (150:1) .. controls ($ (150:1) + (0:0.5) $) and ($ (0:0.3) + (90:0.5) $) .. (0:0.3) .. controls ($ (0:0.3) + (-90:0.5) $) and ($ (-150:1) + (0:0.5) $) .. (-150:1);
\draw [xshift=7.5 cm, yshift=-4 cm, string] (-90:1)  to [bend left=45] (90:1);

\draw [xshift=-1.5 cm, yshift=-8 cm, string] (90:1) to [bend left=45] (-90:1);
\draw [xshift=-1.5 cm, yshift=-8 cm, string] (30:1) to [bend right=45] (210:1);
\draw [xshift=-1.5 cm, yshift=-8 cm, string] (150:1) to [bend right=45] (-30:1);

\draw (0,-8) node {$\longleftrightarrow$};

\draw [xshift=1.5 cm, yshift=-8 cm, string] (90:1) to [bend right=45] (-90:1);
\draw [xshift=1.5 cm, yshift=-8 cm, string] (30:1) to [bend left=45] (210:1);
\draw [xshift=1.5 cm, yshift=-8 cm, string] (150:1) to [bend left=45] (-30:1);

\draw [shorten >=1mm, -to, decorate, decoration={snake,amplitude=.4mm, segment length = 2mm, pre=moveto, pre length = 1mm, post length = 2mm}]
(0,-9.5) -- (0,-10.5);
\draw (0.5,-10) node {$\partial$};

\draw [xshift=-7.5 cm, yshift=-12 cm, string] (90:1) -- (-90:1);
\draw [xshift=-7.5 cm, yshift=-12 cm, string] (30:1) arc (120:240:0.57735);
\draw [xshift=-7.5 cm, yshift=-12 cm, string] (150:1) arc (60:-60:0.57735);

\draw (-6,-12) node {$+$};

\draw [xshift=-4.5 cm, yshift=-12 cm, string] (90:1) -- (-30:1);
\draw [xshift=-4.5 cm, yshift=-12 cm, string] (30:1) -- (210:1);
\draw [xshift=-4.5 cm, yshift=-12 cm, string] (150:1) -- (-90:1);

\draw (-3,-12) node {$+$};

\draw [xshift=-1.5 cm, yshift=-12 cm, string] (90:1) -- (210:1);
\draw [xshift=-1.5 cm, yshift=-12 cm, string] (30:1) -- (-90:1);
\draw [xshift=-1.5 cm, yshift=-12 cm, string] (150:1) -- (-30:1);

\draw (0,-12) node {$\longleftrightarrow$};

\draw [xshift=1.5 cm, yshift=-12 cm, string] (90:1) -- (-30:1);
\draw [xshift=1.5 cm, yshift=-12 cm, string] (30:1) -- (210:1);
\draw [xshift=1.5 cm, yshift=-12 cm, string] (150:1) -- (-90:1);

\draw (3,-12) node {$+$};

\draw [xshift=4.5 cm, yshift=-12 cm, string] (90:1) -- (210:1);
\draw [xshift=4.5 cm, yshift=-12 cm, string] (30:1) -- (-90:1);
\draw [xshift=4.5 cm, yshift=-12 cm, string] (150:1) -- (-30:1);

\draw (6,-12) node {$+$};

\draw [xshift=7.5 cm, yshift=-12 cm, string] (90:1) -- (-90:1);
\draw [xshift=7.5 cm, yshift=-12 cm, string] (30:1) arc (120:240:0.57735);
\draw [xshift=7.5 cm, yshift=-12 cm, string] (150:1) arc (60:-60:0.57735);

\end{tikzpicture}
\caption{Type III string Reidemeister moves and differential.}
\label{fig:type_III_differential}
\end{center}
\end{figure}
\end{proof}

Similarly we may consider string diagrams up to homotopy in general. If $s_0, s_1$ are related by an ambient isotopy or string type II or III Reidemeister move then $\partial s_0$ and $\partial s_1$ are homotopic. If $s_0, s_1$ are related by a string type I Reidemeister move, however, then $\partial s_0$ and $\partial s_1$ are not homotopic: one contains a contractible loop where the other does not. However if we declare all string diagrams with contractible loops to be $0$, then the differential is well-defined and we obtain the following.
\begin{lem}
The map $\partial : \widehat{CS}(\Sigma, F) \rightarrow \widehat{CS}(\Sigma, F)$ is well defined. That is, if string diagrams $s,s'$ without contractible loops are homotopic then $\partial s, \partial s'$ are homotopic (possibly both zero).
\qed
\end{lem}

\subsection{Differential and filtration}

Now that $\partial$ is well-defined, we can see that, given a string diagram $s$ up to spin homotopy, $\partial^2 s$ is well-defined up to spin homotopy and is given by resolving all pairs of crossings in $s$. As each pair of crossings is resolved twice, mod $2$ the result is zero in $CS^\infty(\Sigma,F)$. If we set diagrams with contractible loops equal to zero and consider them up to homotopy, again $\partial^2=0$.
\begin{lem}
\label{lem:diff_squared_is_zero}
The operator $\partial$ is a differential on both $CS^\infty(\Sigma,F)$ and $\widehat{CS}(\Sigma,F)$.
\qed
\end{lem}
It follows that $HS^\infty(\Sigma,F)$ and $\widehat{HS}(\Sigma,F)$ are well-defined. 

We now show that the gradings by intersections become filtrations, i.e. $\partial$ lowers $I^\infty$ and $\widehat{I}$.

\begin{lem}
For $v \in CS^\infty(\Sigma,F)$, $I^\infty(\partial v) \leq I^\infty (v) - 1$.
\end{lem}

\begin{proof}
First, take a string diagram $s$, and suppose $s$ is in general position and has the least number of self-intersections among spin-homotopic string diagrams. Then $\partial s$ is given as a sum $\sum_j s_{j}$, where each $s_j$ is given by resolving a single crossing of $s$, and hence has fewer crossings than $s$. Now $I^\infty(s_j)$ is the least number of crossings in a string diagram spin-homotopic to $s_j$, and hence $I^\infty(s_j) \leq I^\infty(s)-1$. We then have $I^\infty(\partial s) = \max_j I^\infty (s_j) \leq I^\infty (s)-1$.

Taking now a general element $v \in CS(\Sigma,F)$. We may take $v = \sum_i s_i$, where the $s_i$ are in general position, pairwise non-spin-homotopic, and each $s_i$ minimizes self-intersections in its spin homotopy class; so $I^\infty (v) = \max_i I^\infty (s_i)$. Let each $\partial s_i = \sum_j s_{ij}$, so each $I^\infty(s_{ij}) \leq I^\infty(s_i) - 1 \leq I^\infty(v)-1$. Then $I^\infty(\partial v) \leq \max_{i,j} I^\infty(s_{ij}) \leq I^\infty(v)-1$.
\end{proof}

A similar result holds for $\widehat{CS}$, referring everywhere to homotopy rather than spin homotopy, and neglecting any diagrams that have contractible loops.
\begin{lem}
For $v \in \widehat{CS}(\Sigma,F)$, $\widehat{I}(\partial v) \leq \widehat{I}(v)-1$.
\qed
\end{lem}

\subsection{Why these chain complexes?}
\label{sec:why_these}

It may have seemed that we chose two particular types of string diagrams and types of homotopy arbitrarily. We now give a brief argument why.

First, it is natural to consider the collection of all string diagrams. Ideally we would like to consider them up to homotopy, the most general of the ``types of homotopy'' we consider. But the vector space generated by homotopy classes of string diagrams on $(\Sigma,F)$ has no well-defined differential: lemma \ref{lem:homotopic_diff} fails for this vector space. That is, there exist homotopic diagrams $s,s'$ for which $\partial s$ and $\partial s'$ are in no sense homotopic.

For example, consider two string diagrams $s,s'$ related by a type I Reidemeister move (say $s'$ has an extra whirl). So $s,s'$ are homotopic, yet $\partial s$ and $\partial s'$ differ by one term: $\partial s'$ has an extra term with a contractible loop. This is still a nontrivial diagram, and so $\partial s$, $\partial s'$ are not homotopic. Our two chain complexes arise from restricting the type of diagram, or the type of homotopy considered, in a minimal way.

If we want to restrict the type of homotopy considered, examining figure \ref{fig:type_II_differential} leads naturally to the idea that we should consider string diagrams related by a balanced type I string Reidemeister moves as equivalent, and hence to the idea of spin homotopy. And we have seen that, using the finer notion of spin homotopy class, we obtain a well-defined differential and chain complex with filtration, namely $CS^\infty(\Sigma,F)$.

Alternatively, if we want to consider always homotopy classes of diagrams, then the difficulties with type I Reidemeister moves impose the condition that contractible loops should be zero. And indeed we have seen that, considering only string diagrams without contractible loops, we obtain the well-defined complex $\widehat{CS}(\Sigma,F)$. 

The two chain complexes $CS^\infty$ and $\widehat{CS}$, then, are arguably the most natural chain complexes which can be constructed out of curves on a marked surface $(\Sigma,F)$.

\subsection{Definition of $U$}
\label{sec:defn_of_U}

We can now define the $U$ map on $CS^\infty(\Sigma,F)$, for any marked surface $(\Sigma,F)$. 

Given a spin homotopy class of string diagram $\sigma$, we define $U\sigma$ to be obtained from $\sigma$ by adding two anticlockwise whirls. Obviously $U$ does not change the homotopy class of $\sigma$. As two string diagrams realted by a balanced type I Reidemeister string move are spin homotopic, we need not specify where we add the whirls. Likewise we define $U^{-1} \sigma$ by adding two clockwise whirls to $\sigma$. We can see that $U U^{-1} \sigma = U^{-1} U \sigma = \sigma$, and in general $U^i U^j \sigma = U^{i+j} \sigma$ for any $i,j \in \Z$. Indeed we have well-defined maps
\[
U^{\pm 1} : CS^\infty (\Sigma, F) \To CS^\infty(\Sigma,F)
\]
which make $CS^\infty(\Sigma,F)$ into a $\Z_2[U,U^{-1}]$-module.

We also note that $U$ commutes with $\partial$, as shown in diagram \ref{fig:U_commutes}. Note here that it is crucial that $U$ adds an even number of whirls, so the terms obtained by resolving crossings introduced by $U$ cancel. It follows then that $HS^\infty(\Sigma,F)$ also has the structure of a $\Z_2[U,U^{-1}]$-module.

\begin{figure}[h]
\begin{center}
\begin{tikzpicture}[scale=1, 
string/.style={thick, draw=red, postaction={nomorepostaction, decorate, decoration={markings, mark=at position 0.5 with {\arrow{>}}}}}]

\draw (-1,0) node {$\partial$};
\draw [string] (0,-1) .. controls (0,-0.5) and (-0.5,0) .. (-0.5,-0.5) .. controls (-0.5,-1) and (0,-0.5) .. (0,0) .. controls (0,0.5) and (-0.5,1) .. (-0.5,0.5) .. controls (-0.5,0) and (0,0.5) .. (0,1);
\draw (0.5, -1) node {$s$};
\draw (1,0) node {$=$};
\draw [string] (2,-1) -- (2,0) .. controls (2,0.5) and (1.5,1) .. (1.5,0.5) .. controls (1.5,0) and (2,0.5) .. (2,1);
\draw [string] (1.5,-0.5) arc (-90:90:0.25) arc (90:270:0.25);
\draw (2.5,-1) node {$s$};
\draw (3,0) node {$+$};
\draw [string] (4,-1) .. controls (4,-0.5) and (3.5,0) .. (3.5,-0.5) .. controls (3.5,-1) and (4,-0.5) .. (4,0) -- (4,1);
\draw [string] (3.5,0) arc (-90:90:0.25) arc (90:270:0.25);
\draw (4.5,-1) node {$s$};
\draw (5,0) node {$+$};
\draw [string, xshift=6 cm] (0,-1) .. controls (0,-0.5) and (-0.5,0) .. (-0.5,-0.5) .. controls (-0.5,-1) and (0,-0.5) .. (0,0) .. controls (0,0.5) and (-0.5,1) .. (-0.5,0.5) .. controls (-0.5,0) and (0,0.5) .. (0,1);
\draw (6.5,-1) node {$\partial s$};
\draw (7,0) node {$=$};
\draw [string, xshift=8 cm] (0,-1) .. controls (0,-0.5) and (-0.5,0) .. (-0.5,-0.5) .. controls (-0.5,-1) and (0,-0.5) .. (0,0) .. controls (0,0.5) and (-0.5,1) .. (-0.5,0.5) .. controls (-0.5,0) and (0,0.5) .. (0,1);
\draw (8.5,-1) node {$\partial s$};
\end{tikzpicture}
\caption{$\partial U = U \partial$}
\label{fig:U_commutes}
\end{center}
\end{figure}

\section{Non-alternating case}
\label{sec:non-alternating}

We now prove the theorem \ref{thm:zero_thm}: if $F$ is not alternating then $HS^\infty(\Sigma,F) = \widehat{HS}(\Sigma,F) = 0$. 

The proof is based upon a \emph{switching} operation $W$ on a string diagram $s$. Since $F$ is not alternating, there must be two consecutive marked points on $\partial \Sigma$ with the same label, ``in'' or ``out''. Let them be $p$ and $q$. The operation $S$ ``switches'' $s$ between $p$ and $q$ as shown in figure \ref{fig:switching}: it alters $s$ near $p$ and $q$, so that the strand which previously began at $p$, now begins at $q$; and vice versa, the strand which previously began at $q$, now begins at $p$; introducing precisely one new crossing in the process. 

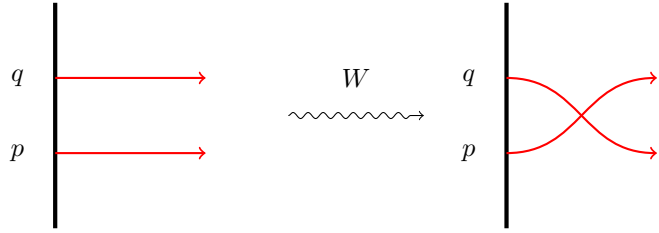
\begin{figure}[ht]
\begin{center}
\begin{tikzpicture}[
scale=1, 
string/.style={thick, draw=red, -to},
boundary/.style={ultra thick}]

\draw [boundary] (0,0) -- (0,3);
\draw [string] (0,1) -- (2,1);
\draw [string] (0,2) -- (2,2);
\draw (-0.5,1) node {$p$};
\draw (-0.5,2) node {$q$};

\draw [shorten >=1mm, -to, decorate, decoration={snake,amplitude=.4mm, segment length = 2mm, pre=moveto, pre length = 1mm, post length = 2mm}]
(3,1.5) -- (5,1.5);
\draw (4,2) node {$W$};

\draw [boundary] (6,0) -- (6,3);
\draw [string] (6,1) .. controls (7,1) and (7,2) .. (8,2);
\draw [string] (6,2) .. controls (7,2) and (7,1) .. (8,1);
\draw (5.5,1) node {$p$};
\draw (5.5,2) node {$q$};

\end{tikzpicture}
\caption{Switching operation.} \label{fig:switching}
\end{center}
\end{figure}

If $s$ and $s'$ are ambient isotopic, then clearly $Ws$ and $Ws'$ are also. Similarly, if $s,s'$ are related by type I, II or III string Reidemeister move, then so are $Ws$ and $Ws'$. So $W$ certainly gives a well-defined operation on string diagrams up to homotopy, regular homotopy or spin homotopy. Further if $s$ is without contractible loops then so too is $Ws$. So we obtain well-defined linear maps on $CS^\infty$ and $\widehat{CS}$; in a minor abuse of notation we denote both by $W$.
\[
W \; : \; CS^\infty(\Sigma,F) \To CS^\infty(\Sigma,F), \quad
\widehat{CS}(\Sigma, F) \To \widehat{CS}(\Sigma,F).
\]

Given a string diagram $s$, consider $\partial W s$. First $Ws$ is obtained from $s$ by the switching operation near $p$ and $q$, introducing one more crossing, and then $\partial Ws$ is obtained from $Ws$ by resolving each crossing and summing the resulting diagrams. (If we are working in $CS^\infty$, this is a sum of diagrams up to spin homotopy; if in $\widehat{CS}$, up to homotopy, and setting contractible loops to zero.)

The diagram obtained by resolving the new intersection point in $Ws$ is just $s$. The diagrams obtained from resolving the other intersections points are just the diagrams in $\partial s$, but with the switching $W$ then applied. Thus we have $\partial W s = s + W \partial s$; see figure \ref{fig:chain_homotopy}.

\begin{figure}[h]
\begin{center}
\begin{tikzpicture}[
scale=1, 
string/.style={thick, draw=red, -to},
boundary/.style={ultra thick}]

\draw (-1,1.5) node {$\partial$};
\draw [boundary] (0,0) -- (0,3);
\draw [rounded corners=8pt] (0,3) -- (3,3) -- (3,0) -- (0,0);
\draw [string] (0,1) .. controls (0.5,1) and (0.5,2) .. (1,2);
\draw [string] (0,2) .. controls (0.5,2) and (0.5,1) .. (1,1);
\draw (2,1.5) node {$s$};

\draw (3.5,1.5) node {$=$};

\draw [xshift=4 cm, boundary] (0,0) -- (0,3);
\draw [xshift=4 cm, rounded corners=8pt] (0,3) -- (3,3) -- (3,0) -- (0,0);
\draw [xshift=4 cm, string] (0,1) -- (1,1);
\draw [xshift=4 cm, string] (0,2) -- (1,2);
\draw [xshift=4 cm] (2,1.5) node {$s$};

\draw (7.5,1.5) node {$+$};

\draw [xshift=8 cm, boundary] (0,0) -- (0,3);
\draw [xshift=8 cm, rounded corners=8pt] (0,3) -- (3,3) -- (3,0) -- (0,0);
\draw [xshift=8 cm, string] (0,1) .. controls (0.5,1) and (0.5,2) .. (1,2);
\draw [xshift=8 cm, string] (0,2) .. controls (0.5,2) and (0.5,1) .. (1,1);
\draw [xshift=8 cm] (2,1.5) node {$\partial s$};

\draw (1.5,-1) node {$\partial Ws$};
\draw (3.5,-1) node {$=$};
\draw (5.5,-1) node {$s$};
\draw (7.5,-1) node {$+$};
\draw (9.5,-1) node {$W \partial s$};

\end{tikzpicture}
\caption{The operation $S$ is a chain homotopy.} \label{fig:chain_homotopy}
\end{center}
\end{figure}

This equation holds both in $CS^\infty(\Sigma,F)$ and $\widehat{CS}(\Sigma,F)$. We may therefore (always mod 2) write
\[
\partial W + W \partial = 1.
\]
That is, $W$ is a chain homotopy between the chain maps $1$ and $0$ on $CS^\infty(\Sigma,F)$ or $\widehat{CS}(\Sigma,F)$. It follows that
\[
HS^\infty(\Sigma,F) = \widehat{HS}(\Sigma,F) = 0
\]
as desired. 

Explicitly, if $x \in \ker \partial$, then $\partial Wx + W \partial x = x$ so that $x = \partial (Wx)$ is a boundary.

Note that this proof works even if $F$ does not have an even number of points on each boundary component; all we require is two consecutive points of $F$ of the same sign somewhere on $\partial \Sigma$.

Henceforth we will assume all markings are alternating, so that $(\Sigma,F)$ always has the structure of a sutured background.

\section{The generalised Euler class on discs}
\label{sec:generalised_euler_class}

A set of sutures $\Gamma$ on a sutured background $(\Sigma,F)$ is has an \emph{Euler class} given by $e(\Gamma) = \chi(R_+) - \chi(R_-)$. We now generalise this notion to the Euler class of a string diagram on an alternating disc $(D^2,F)$.

The idea is that the Euler class $e(\Gamma)$ can be described in terms of the \emph{curvature} of the curves of $\Gamma$ with respect to a standard metric on $D^2$. 

Note that this section applies only to discs. Complications arise when trying to apply these ideas to more general surfaces.

\subsection{Euler class of sutures via curvature}

We first reinterpret the Euler class for sutures on discs in terms of curvature. Let $\Gamma$ be a set of sutures on the disc sutured background $(D, F)$, where $F$ consists of $2n$ points alternating in sign around $\partial D$. Consider $D$ as the unit disc in the Euclidean plane, with the $2n$ points of $F$ spaced equally around the unit circle, and by ambient isotopy assume that all sutures intersect $\partial D$ at right angles. 

Let $\gamma$ be a suture, i.e. a component of $\Gamma$; so $\gamma$ is either a properly embedded arc in $D$ joining two points of $F$, or is an embedded closed curve. Suppose $\gamma$ is traversed at unit speed, and consider its velocity vector; it turns through some total angle $k$, measured anticlockwise. Note that if $\gamma$ is a closed curve then, being embedded, $k= \pm 2\pi$. If $\gamma$ is an arc, as the endpoints $F$ are equally spaced and $\gamma$ meets $\partial D$ at right angles, $k$ must be an integer multiple of $2\pi/2n = \pi/n$. (More precisely, taking into account labels on sutures: if $n$ is odd then $k$ is an integer multiple of $2\pi/n$; if $n$ is even then $k$ is of the form $\frac{(2l+1)\pi}{n}$ for some integer $l$.) This $k$ is the total curvature of $\gamma$.

Let the components of $\Gamma$ be $\gamma_1, \ldots, \gamma_M$, and let $\gamma_i$ have curvature $k_i$.
\begin{lem}
$e(\Gamma) = \frac{1}{\pi} \sum_{i=1}^M k_i$.
\end{lem}

\begin{proof}
First suppose $\Gamma$ has no closed curve components, so the number of sutures $M=n$ and $R_+, R_-$ both consist of discs. Consider a disc component of $R_+$: its boundary consists of sutures $\gamma_j$ (traversed in the direction of $\gamma_j$), and arcs of $\partial D$ of curvature $\pi/n$, which meet sutures at right angles. Moreover, as we traverse the boundary of all of $R_+$, we traverse each suture of $\Gamma$, precisely half of $\partial D$, and precisely $2n$ right angles. As each component of $R_+$ has a single full turn around its boundary, the number of components of $R_+$ is
\[
\frac{1}{2\pi} \left( \left( \sum_{i=1}^n k_i \right) + \pi + 2n \frac{\pi}{2}
\right) = \frac{1}{2\pi} \left( \sum_{i=1}^n k_i \right) + \frac{n+1}{2}.
\]
Similarly, considering $\partial R_-$, the number of components of $R_-$ is
\[
- \frac{1}{2\pi} \left( \sum_{i=1}^n k_i \right) + \frac{n+1}{2}.
\]
As all components of $R_\pm$ are discs, we have
\[
e(\Gamma) = \chi(R_+) - \chi(R_-) = \frac{1}{\pi} \sum_{i=1}^n k_i
\]
as desired.

Now suppose $\Gamma$ also contains closed curve components. Adding an anticlockwise closed curve suture adds an extra disc region to $R_+$ and removes a disc from a region of $R_-$. Thus $\chi(R_+)$ increases by $1$ and $\chi(R_-)$ decreases by $1$, so $e(\Gamma)$ increases by $2$. The new curve has curvature $2\pi$, $\frac{1}{\pi} \sum k_i$ also increases by $2$. Similarly for a clockwise suture, both $e(\Gamma)$ and $\frac{1}{\pi} \sum k_i$ decrease by $2$. Any set of sutures on $(D,F)$ can be constructed from sutures without closed curves by repeatedly adding sutures in this way.
\end{proof}

In fact a more general result is possible; there is no necessity to restrict to a standard round metric with points of $F$ evenly spaced, but this is all we need.

\subsection{Generalised Euler class of string diagrams}
\label{sec:gen_Euler_class}

Consider now a general string diagram $s$ on an alternating disc $(D,F)$. Again consider $D$ as the unit disc in the Euclidean plane, with the $2n$ points of $F$ equally spaced around the unit circle; again require curves of $s$ to meet $\partial D$ at right angles. Let the curves of $s$ be $\sigma_1, \ldots, \sigma_M$; each $\sigma_i$ is oriented and has a total curvature $k_i$.

\begin{defn}
The \emph{generalised Euler class} $e(s)$ of $s$ is $\frac{1}{\pi} \sum_{i=1}^M k_i$.
\end{defn}

Obviously this generalises the Euler class of sutures. It is not difficult to see that the total curvature of $s$ is unchanged by
\begin{enumerate}
\item
ambient isotopy of $s$ on $\Sigma$ (always requiring $s$ to meet $\partial \Sigma$ at right angles);
\item
type II string Reidemeister moves;
\item
type III string Reidemeister moves;
\item
balanced type I Reidemeister moves;
\item
resolving a crossing.
\end{enumerate}

It follows that any string diagram obtained by successively resolving crossings and performing spin homotopies has the same total curvature as $s$. In particular we have proved the following.
\begin{prop}
Any set of sutures $\Gamma$ obtained by resolving crossings and performing spin homotopies of a string diagram $s$ satisfies $e(\Gamma) = e(s)$.
\qed
\end{prop}

Resolving all the crossings of a string diagram $s$ on $(D,F)$ will result in a string diagram without crossings, which is not necessarily a set of sutures, but which is homotopic to a set of sutures.  For instance consider resolving the crossing in a ``whirl'' created from a type I Reidemeister move; it does not form a set of sutures, but the contractible loop can be homotoped across a string to give sutures. Note this statement is only true for discs.

Note that ``adding a whirl'' by a type I string Reidemeister move changes the Euler class by $\pm 1$ respectively as the whirl is anticlockwise or clockwise; a change of $+1$ and $-1$ cancel out in a balanced type I move.

Note also that, from this definition, it is clear that $\partial$ preserves $e$: resolving a crossing replaces an intersection (where curves can be assume straight and intersecting at right angles) with two curving segments, of curvature $-\pi/2$ and $\pi/2$, contributing $0$ to total curvature. It follows that the direct sum decomposition $CS^\infty(D^2,F) = \bigoplus_e CS_e^\infty (D^2,F)$ also gives a direct sum decomposition in homology, $HS^\infty = \bigoplus_e HS_e^\infty (D^2, F)$.

The above leads us to an alternative definition of spin homotopy.
\begin{prop}
\label{prop:spin_homotopy}
Two string diagrams on $(D,F)$ are spin homotopic if and only if they are homotopic and have the same generalised Euler class.
\end{prop}

\begin{proof}
If string diagrams are spin homotopic then, by definition, they are related by ambient isotopies and type II, type III and balanced type I string Reidemeister moves. All these moves result in homotopic string diagrams with the same Euler class.

Conversely, if $s_0, s_1$ are homotopic and have the same Euler class, they are related by ambient isotopies and type I, II and III Reidemeister moves. However of these, only type I moves change Euler class. As $s_0, s_1$ have the same Euler class, the number of type I moves which increase and decrease Euler class must be equal. 

It remains to show that the homotopy between $s_0$ and $s_1$ can be achieved with positive and negative type I Reidemeister moves occur in pairs at the same time. There are $4$ variants of the type I move: the positive moves consist of adding an anticlockwise whirl or deleting a clockwise whirl; the negative moves consist of adding a clockwise whirl or deleting an anticlockwise whirl.

We note first that a  ``deletion of whirl'' type I move can be achieved by performing the ``addition of whirl'' type I move of the same sign, followed by a regular homotopy. So we can assume that all type I moves consist of additions of whirls; and hence the number of additions of anticlockwise whirls is equal to the number of additions of clockwise whirls. 

We also note that a whirl added at time $t$ in a homotopy, can be added at a time earlier than $t$, and then carried forward to time $t$. So type I moves can indeed be added in balanced pairs. Hence the homotopy can be achieved by using balanced type I Reidemeister moves, and $s_0, s_1$ are spin homotopic.
\end{proof}

We can now see that the \emph{relative winding} of a string diagram $s_1$ with respect to a homotopic string diagram $s_0$, defined in section \ref{sec:spin_homotopy}, is just $e(s_1) - e(s_0)$. And we can see that a homotopy class of string diagrams splits into a countable infinity of spin homotopy classes, indexed precisely by Euler class.

Note the effect of $U$ on $e$: adding two anticlockwise whirls adjusts $e$ by $4$, and adding two clockwise whirls adjusts $e$ by $-4$. So the maps $U, U^{-1}$ on $CS^\infty(D^2, F)$ restrict to the summands $CS^\infty_e(D^2, F)$ generated by string diagrams with fixed Euler class $e$ as
\[
U^{\pm 1} : CS^\infty_e (D^2, F) \To CS^\infty_{e \pm 4} (D^2, F).
\]
Since $U$ commutes with $\partial$ this also applies to homology:
\[
U^{\pm 1} : HS_e^\infty (D^2, F) \To HS_{e \pm 4}^\infty (D^2, F).
\]

\section{Creation and annihilation}
\label{sec:creation_annihilation}

Our proof proceeds by showing that various parts of the structure developed in \cite{Me09Paper} apply in the present situation. In particular, we will use \emph{creation and annihilation operators} similar to the operators defined in that paper.

We will only need creation and annihilation operators on discs, but there is no more difficulty in defining these operators on general surfaces. After giving the definition in general, we give our main proofs, which apply only to discs.

\subsection{Creation operators}

We define a \emph{creation operator} to take a string diagram on a sutured background $(\Sigma,F)$ and insert an outermost string, not intersecting any others, at a specified location on $\partial \Sigma$, with the new string in a specified orientation. 

More precisely, given an alternating marking $F$ on $\Sigma$, we consider an alternating marking $F'$ obtained from $F$ by adding two points $f_{in}, f_{out}$, lying in the same component of $\partial \Sigma \backslash F$, respectively labelled ``in'' and ``out''. We consider $CS^\infty$ and $\widehat{CS}$ cases separately, though the pictures are similar. As $F, F'$ are required to be alternating, there is precisely one way to insert $f_{in}, f_{out}$ between two consecutive points of $F$. (One can easily, however, define similar operators when $F$ is not alternating.)

The creation operator
\[
\bar{a}^*_{F,F'} \; : \; CS^\infty(\Sigma,F) \To CS^\infty(\Sigma,F')
\]
takes a spin homotopy class $\bar{s}$ of a string diagram $s$ on $(\Sigma,F)$, and inserts an extra string from $f_{in}$ to $f_{out}$, not intersecting itself or any other strands. The result $\bar{a}^*_{F,F'} \bar{s}$ is well defined up to spin homotopy. Being defined on a basis of $CS^\infty(\Sigma,F)$, $\bar{a}^*_{F,F}$ extends to a linear map on $CS^\infty$.

Similarly, the creation operator
\[
\widehat{a}^*_{F,F'} \; : \; \widehat{CS}(\Sigma,F) \To \widehat{CS}(\Sigma,F')
\]
takes a homotopy class $\widehat{s}$ of a string diagram $s$ without contractible loops on $(\Sigma,F)$, and inserts an extra string in the same way, giving a result $\widehat{a}^*_{F,F'} \widehat{s}$ without contractible loops and well defined up to homotopy. It extends to a linear map on $\widehat{CS}$. See figure \ref{fig:creation_operator}.

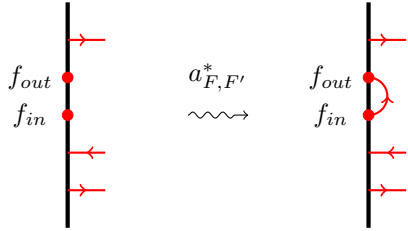
\begin{figure}[ht]
\begin{center}
\begin{tikzpicture}[
scale=1, 
string/.style={thick, draw=red, postaction={nomorepostaction, decorate, decoration={markings, mark=at position 0.5 with {\arrow{>}}}}},
boundary/.style={ultra thick}]

\draw [boundary] (0,0) -- (0,3);
\draw [string] (0,0.5) -- (0.5,0.5);
\draw [string] (0.5,1) -- (0,1);
\draw [string] (0,2.5) -- (0.5,2.5);
\draw (-0.5,1.5) node {$f_{in}$};
\draw (-0.5,2) node {$f_{out}$};
\fill [draw=red, fill=red] (0,1.5) circle (2pt);
\fill [draw=red, fill=red] (0,2) circle (2pt);

\draw [shorten >=1mm, -to, decorate, decoration={snake,amplitude=.4mm, segment length = 2mm, pre=moveto, pre length = 1mm, post length = 2mm}]
(1.5,1.5) -- (2.5,1.5);
\draw (2,2) node {$a_{F,F'}^*$};

\draw [boundary] (4,0) -- (4,3);
\draw [string] (4,0.5) -- (4.5,0.5);
\draw [string] (4.5,1) -- (4,1);
\draw [string] (4,1.5) arc (-90:90:0.25);
\draw [string] (4,2.5) -- (4.5,2.5);
\draw (3.5,1.5) node {$f_{in}$};
\draw (3.5,2) node {$f_{out}$};
\fill [draw=red, fill=red] (4,1.5) circle (2pt);
\fill [draw=red, fill=red] (4,2) circle (2pt);

\end{tikzpicture}
\caption{Creation operator.} \label{fig:creation_operator}
\end{center}
\end{figure}

Note that if $s_0, s_1$ are string diagrams on $(\Sigma,F)$ which are not spin homotopic, then inserting the extra string in a creation operator results in string diagrams which are not spin homotopic. It follows that any $\widehat{a}_{F,F'}^*$ is injective. Similarly we see that any $\bar{a}^*_{F,F'}$ is injective.

We next consider the effect of a creation operator followed by the differential, i.e. $\partial a^*_{F,F'}$. Again the pictures are similar in the $CS^\infty$ and $\widehat{CS}$ cases.

Given a spin homotopy class of string diagram $\bar{s} \in CS^\infty(\Sigma,F)$, we add an additional string to obtain $\bar{a}^*_{F,F'} \bar{s}$. Consider resolving (single) crossings in this string diagram. The resolutions are precisely those of $s$, but with the extra strand from $f_{in}$ to $f_{out}$ adjoined. Thus
\[
\partial \bar{a}^*_{F,F'} \bar{s} = \bar{a}^*_{F,F'} \partial \bar{s}.
\]
In fact, this is true at the level of ambient isotopy classes.

Similarly, given a regular homotopy class $\widehat{s}$ of a string diagram $s$ without contractible loops, resolving crossings in $\widehat{a}^*_{F,F'} \widehat{s}$ gives precisely the string diagrams of $\partial \widehat{s}$ with the extra strand adjoined; and any resolution which creates a contractible loop also created a contractible loop in $\partial \widehat{s}$. Thus
\[
\partial \widehat{a}^*_{F,F'} \widehat{s} = \widehat{a}^*_{F,F'} \partial \widehat{s}.
\]

Thus $\partial \bar{a}^*_{F,F'} = \bar{a}^*_{F,F'} \partial$ and $\partial \widehat{a}^*_{F,F'} = \widehat{a}^*_{F,F'} \partial$. We have proved the following.
\begin{lem}
The creation operators $\bar{a}^*_{F,F'}, \widehat{a}^*_{F,F'}$ are chain maps, hence define maps
\[
\bar{a}^*_{F,F'} \; : \; HS^\infty (\Sigma,F) \To HS^\infty (\Sigma,F'),
\quad
\widehat{a}^*_{F,F'} \; : \; \widehat{HS}(\Sigma,F) \To \widehat{HS}(\Sigma,F').
\]
\qed
\end{lem}

\subsection{Annihilation operators}

In a similar fashion we may define \emph{annihilation} operators. An annihilation operator takes a string diagram and ``closes off'' two consecutive points of $F$, as shown in figure \ref{fig:annihilation_operator}, to give a string diagram with fewer arcs (and maybe a new closed curve).

More precisely, given an alternating marking $F$ on $\Sigma$, consider a marking $F'$ obtained from $F$ by \emph{removing} two consecutive points $f_{in}, f_{out}$ of $F$, respectively labelled ``in'' and ``out''. Note that for $F'$ to be a valid marking there must be at least $4$ points of $F$ on the boundary component of $f_{in}$ and $f_{out}$. Also note that $F'$ is necessarily alternating; although one could easily define a similar operation on a non-alternating $F$, whenever two consecutive points of $F$ have opposite directions. Again we consider $CS^\infty$ and $\widehat{CS}$ cases separately but pictures are similar.

The annihilation operator
\[
\bar{a}_{F,F'} \; : \; CS^\infty(\Sigma,F) \To CS^\infty(\Sigma,F')
\]
takes a spin homotopy class $\bar{s}$ of a string diagram $s$ on $(\Sigma,F)$ and joins the strings previously ending at $f_{in}, f_{out}$, without introducing any new intersections of strings. The result $\bar{a}_{F,F'}$ is well defined up to spin homotopy and we linearly extend to define $\bar{a}_{F,F'}$ on $CS^\infty(\Sigma,F)$. Note that if the strings ending at $f_{in}, f_{out}$ in $s$ are distinct then $\bar{a}_{F,F'} \bar{s}$ has one fewer arc component that $s$; while if a single string has endpoints at $f_{in},f_{out}$ then $\bar{a}_{F,F'} \bar{s}$ has two fewer arc components than $s$ but one more closed curve component.

Similarly, the annihilation operator
\[
\widehat{a}_{F,F'} \; : \; \widehat{CS}(\Sigma,F) \To \widehat{CS}(\Sigma,F')
\]
takes a homotopy class $\widehat{s}$ of a string diagram $s$ without contractible loops on $(\Sigma,F)$, and joins strings in the same way, giving a result well-defined up to homotopy. If joining the strings results in a contractible loop then we regard the result as zero in $\widehat{CS}(\Sigma,F')$. We extend linearly to define $\widehat{a}_{F,F'}$ on $\widehat{CS}(\Sigma,F)$. See figure \ref{fig:annihilation_operator}.

\begin{figure}[ht]
\begin{center}
\begin{tikzpicture}[
scale=1, 
string/.style={thick, draw=red, postaction={nomorepostaction, decorate, decoration={markings, mark=at position 0.5 with {\arrow{>}}}}},
boundary/.style={ultra thick}]

\draw [boundary] (0,0) -- (0,3);
\draw [string] (0,0.5) -- (0.5,0.5);
\draw [string] (0.5,1) -- (0,1);
\draw [string] (0,2.5) -- (0.5,2.5);
\draw (-0.5,1.5) node {$f_{in}$};
\draw (-0.5,2) node {$f_{out}$};
\fill [draw=red, fill=red] (0,1.5) circle (2pt);
\fill [draw=red, fill=red] (0,2) circle (2pt);
\draw [string] (0,1.5) -- (1,1.5);
\draw [string] (1,2) -- (0,2);

\draw [shorten >=1mm, -to, decorate, decoration={snake,amplitude=.4mm, segment length = 2mm, pre=moveto, pre length = 1mm, post length = 2mm}]
(1.5,1.5) -- (2.5,1.5);
\draw (2,2) node {$a_{F,F'}$};

\draw [boundary] (4,0) -- (4,3);
\draw [string] (4,0.5) -- (4.5,0.5);
\draw [string] (4.5,1) -- (4,1);
\draw [string] (5,2) -- (4.5,2) arc (90:270:0.25) -- (5,1.5);
\draw [string] (4,2.5) -- (4.5,2.5);
\draw (3.5,1.5) node {$f_{in}$};
\draw (3.5,2) node {$f_{out}$};
\fill [draw=red, fill=red] (4,1.5) circle (2pt);
\fill [draw=red, fill=red] (4,2) circle (2pt);

\end{tikzpicture}
\caption{Annihilation operator.} \label{fig:annihilation_operator}
\end{center}
\end{figure}
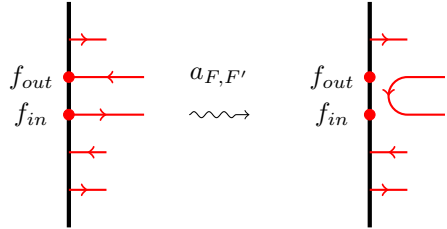

Consider the effect of an annihilation operator fillowed by the differential, i.e. $\partial a_{F,F'}$. As the annihilation operation introduces no new crossings, it commutes with the differential, in both the $CS^\infty$ and $\widehat{CS}$ cases, similarly to creation operators.
\begin{lem}
The annihilation operators $\bar{a}_{F,F'}$, $\widehat{a}_{F,F'}$ are chain maps, hence define maps
\[
\bar{a}_{F,F'} \; : \; HS^\infty (\Sigma,F) \To HS^\infty(\Sigma,F'),
\quad
\widehat{a}_{F,F'} \; : \; \widehat{HS}(\Sigma,F) \To \widehat{HS}(\Sigma,F').
\]
\qed
\end{lem}

\section{Homology computation for discs}
\label{sec:homology_computation}

We now compute $\widehat{HS}(D^2, F)$ when $F$ is an alternating marking on $D^2$, proving theorem \ref{thm:main_thm}, and hence theorem \ref{thm:rough_main_thm2}. Let $|F|=2n$; for convenience we will write $F_n$ for the alternating marking on the disc with $2n$ points. The proof will be by induction on $n$.

\subsection{Base case}

\begin{lem}
\label{lem:base_case}
$\widehat{HS}(D^2,F_1) \cong \Z_2$. The single summand lies in intersection grading $0$ and is generated by the string diagram consisting of a single arc.
\end{lem}

\begin{proof}
As $\widehat{CS}$ only considers string diagrams without contractible loops, any string diagram in $\widehat{CS}(D^2, F_1)$ consists of a single arc joining the two points of $F_1$. As $\widehat{CS}$ considers string diagrams up to homotopy, such a string diagram is equivalent to a single properly embedded arc. Hence $\widehat{CS}(D^2, F_1)$ is spanned over $\Z_2$ by this single homotopy class of diagram, which has intersection number $0$. We have $\partial = 0$, so $\widehat{HS}$ is as claimed.
\end{proof}

Following \cite{Me09Paper}, we denote the nonzero element of $\widehat{HS}(D^2,F_1)$ as $v_\emptyset$ and call it the \emph{vacuum}. 

\subsection{Building a basis}
\label{sec:building_a_basis}

On each $(D^2, F_n)$ we will select once and for all a basepoint in $F_n$, labelled ``in''. We can then consider two specific creation operators on $(D^2, F_n)$, creating new strands in the two sites adjacent to the basepoint; and two specific annihilation operators, annihilating at the two sites which include the basepoint. After creating or annihilating at these sites, the basepoints are positioned as shown in figure \ref{fig:a_and_a_star}.

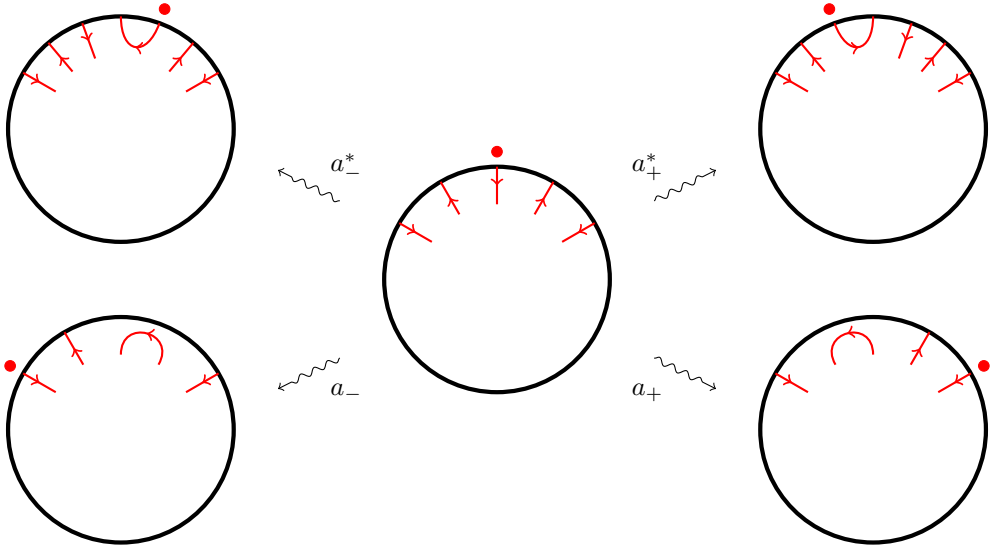
\begin{figure}[ht]
\begin{center}
\begin{tikzpicture}[
scale=1, 
string/.style={thick, draw=red, postaction={nomorepostaction, decorate, decoration={markings, mark=at position 0.5 with {\arrow{>}}}}},
boundary/.style={ultra thick}]

\draw [boundary] (0,0) circle (1.5 cm);
\draw [string] (30:1.5) -- (30:1);
\draw [string] (60:1) -- (60:1.5);
\draw [string] (90:1.5) -- (90:1);
\draw [string] (120:1) -- (120:1.5);
\draw [string] (150:1.5) -- (150:1);
\fill [draw=red, fill=red] (90:1.7) circle (2pt);

\draw [shorten >=1mm, -to, decorate, decoration={snake,amplitude=.4mm, segment length = 2mm, pre=moveto, pre length = 1mm, post length = 2mm}]
(2,1) -- (3,1.5);
\draw (2,1.5) node {$a_+^*$};

\draw [xshift = 5 cm, yshift = 2 cm, boundary] (0,0) circle (1.5 cm);
\draw [xshift = 5 cm, yshift = 2 cm, string] (30:1.5) -- (30:1);
\draw [xshift = 5 cm, yshift = 2 cm, string] (50:1) -- (50:1.5);
\draw [xshift = 5 cm, yshift = 2 cm, string] (70:1.5) -- (70:1);
\draw [xshift = 5 cm, yshift = 2 cm, string] (110:1.5) .. controls (110:1) and (90:1) .. (90:1.5);
\draw [xshift = 5 cm, yshift = 2 cm, string] (130:1) -- (130:1.5);
\draw [xshift = 5 cm, yshift = 2 cm, string] (150:1.5) -- (150:1);
\fill [xshift = 5 cm, yshift = 2 cm, draw=red, fill=red] (110:1.7) circle (2pt);

\draw [shorten >=1mm, -to, decorate, decoration={snake,amplitude=.4mm, segment length = 2mm, pre=moveto, pre length = 1mm, post length = 2mm}]
(2,-1) -- (3,-1.5);
\draw (2,-1.5) node {$a_+$};

\draw [xshift = 5 cm, yshift = -2 cm, boundary] (0,0) circle (1.5 cm);
\draw [xshift = 5 cm, yshift = -2 cm, string] (30:1.5) -- (30:1);
\draw [xshift = 5 cm, yshift = -2 cm, string] (60:1) -- (60:1.5);
\draw [xshift = 5 cm, yshift = -2 cm, string] (90:1) .. controls (90:1.5) and (120:1.5) .. (120:1);
\draw [xshift = 5 cm, yshift = -2 cm, string] (150:1.5) -- (150:1);
\fill [xshift = 5 cm, yshift = -2 cm, draw=red, fill=red] (30:1.7) circle (2pt);

\draw [shorten >=1mm, -to, decorate, decoration={snake,amplitude=.4mm, segment length = 2mm, pre=moveto, pre length = 1mm, post length = 2mm}]
(-2,1) -- (-3,1.5);
\draw (-2,1.5) node {$a_-^*$};

\draw [xshift = -5 cm, yshift = 2 cm, boundary] (0,0) circle (1.5 cm);
\draw [xshift = -5 cm, yshift = 2 cm, string] (30:1.5) -- (30:1);
\draw [xshift = -5 cm, yshift = 2 cm, string] (50:1) -- (50:1.5);
\draw [xshift = -5 cm, yshift = 2 cm, string] (70:1.5) .. controls (70:1) and (90:1) .. (90:1.5);
\draw [xshift = -5 cm, yshift = 2 cm, string] (110:1.5) -- (110:1);
\draw [xshift = -5 cm, yshift = 2 cm, string] (130:1) -- (130:1.5);
\draw [xshift = -5 cm, yshift = 2 cm, string] (150:1.5) -- (150:1);
\fill [xshift = -5 cm, yshift = 2 cm, draw=red, fill=red] (70:1.7) circle (2pt);

\draw [shorten >=1mm, -to, decorate, decoration={snake,amplitude=.4mm, segment length = 2mm, pre=moveto, pre length = 1mm, post length = 2mm}]
(-2,-1) -- (-3,-1.5);
\draw (-2,-1.5) node {$a_-$};

\draw [xshift = -5 cm, yshift = -2 cm, boundary] (0,0) circle (1.5 cm);
\draw [xshift = -5 cm, yshift = -2 cm, string] (30:1.5) -- (30:1);
\draw [xshift = -5 cm, yshift = -2 cm, string] (60:1) .. controls (60:1.5) and (90:1.5) .. (90:1);
\draw [xshift = -5 cm, yshift = -2 cm, string] (120:1) -- (120:1.5);
\draw [xshift = -5 cm, yshift = -2 cm, string] (150:1.5) -- (150:1);
\fill [xshift = -5 cm, yshift = -2 cm, draw=red, fill=red] (150:1.7) circle (2pt);

\end{tikzpicture}
\caption{Creation and annihilation operators $a^*_\pm, a_\pm$. Basepoints are denoted by a dot.} \label{fig:a_and_a_star}
\end{center}
\end{figure}

Thus we obtain on each $\widehat{CS}(D^2,F_n)$ two annihilation operators $a_\pm$ and two creation operators $a^*_\pm$; here we follow the notation of \cite{Me10_Sutured_TQFT}, which is different from that of \cite{Me09Paper} (where they were called $A_\pm, B_\pm$). Being chain maps, these operators descend to homology.
\[
a_\pm \; : \; \widehat{HS}(D^2, F_n) \To \widehat{HS}(D^2, F_{n-1}),
\quad
a^*_\pm \; : \; \widehat{HS}(D^2, F_n) \To \widehat{HS} (D^2, F_{n+1}).
\]
These operators satisfy the relations
\[
a_- a^*_- = a_+ a^*_+ = 1, \quad a_- a^*_+ = a_+ a^*_- = 0.
\]
In particular, each creation $a^*_\pm$ is injective, with partial inverse $a_\pm$.

Still following \cite{Me09Paper}, for any word $w$ of length $n$ on the symbols $\{-,+\}$, we may compose the corresponding creation operators to obtain a creation operator $a^*_w$. Then we define $v_w = a^*_w v_\emptyset \in \widehat{HS}(D^2, F_{n+1})$; this generalises the notation of $v_\emptyset$ by regarding $\emptyset$ as the empty word, of length $0$. As there are $2^n$ words of length $n$ on $\{-,+\}$, we obtain $2^n$ distinguished elements (although we do not yet know they are distinct) in each $\widehat{HS}(D^2, F_{n+1})$. These diagrams are described at length in \cite{Me09Paper}. We will show below that they form a basis.

\begin{lem}
\label{lem:linearly_independent}
The $2^n$ elements of the form $v_w$ in $\widehat{HS}(D^2, F_{n+1})$ are linearly independent.
\end{lem}
(In particular, the $v_w$ are distinct!)

\begin{proof}
This proof appears in \cite{Me09Paper}. Suppose some nontrivial linear combination $\sum_i v_{w_i} = 0$, where the $w_i$ are distinct words of length $n$. For each word $w_i$, there is a sequence of annihilation operators which undo the creation operators used to create $v_{w_i}$; this sequence of annihilation operators sends $v_{w_i} \mapsto v_\emptyset$ but annihilates every other $v_{w_i}$ to $0$. Applying this annihilation operator to $\sum_i v_{w_i} = 0$ then gives $v_\emptyset = 0$, a contradiction.
\end{proof}

\subsection{The crossed wires lemma}

The following lemma is the technical key to the present computation. It applies more generally than to discs, and so we state it generally. It only requires \emph{regular} homotopy, and it works whether or not we disregard contractible loops. We will state it for the more general case of $CS^\infty$, and then for the case immediately at hand, of $\widehat{CS}$.

The lemma applies in a situation where we have two creation operators $a_\pm^*$ which insert strings at adjacent sites, from an alternating marked surface $(\Sigma,F)$ to $(\Sigma,F')$; $F'$ is obtained from $F$ by adding two adjacent points. The strings created by $a_\pm^*$ have endpoints which together form $3$ consecutive points of $F'$: let them be $f_{-1}, f_0, f_1$ in order around $\partial \Sigma$. This generalises the operators above in figure \ref{fig:a_and_a_star}.

\begin{lem}[Crossed wires lemma]
\label{lem:decomposition}
Let $\Sigma,F, F'$ and
\[
a_\pm^* \; : \; CS^\infty(\Sigma,F) \To CS^\infty(\Sigma,F')
\]
be as above. Suppose $x \in CS^\infty(\Sigma, F')$ satisfies $\partial x = 0$. Then there exist $y,z \in CS^\infty(\Sigma,F)$ and $u \in CS^\infty(\Sigma,F')$ such that
\[
\partial y = \partial z = 0
\quad \text{and} \quad
x = a_-^* y + a_+^* z + \partial u.
\]
\end{lem}

To prove this lemma, we will need a certain switching operation $B$ on string diagrams on $(\Sigma, F)$, which switches $f_{-1}$ and $f_1$, so ``crosses 3 wires''; hence the name of the lemma. More precisely, given a string diagram $s$ on $(\Sigma,F)$, we make a local modification near the arc of $\partial \Sigma$ connecting $f_{-1},f_0,f_1$. The arc of $s$ which ran to $f_1$, we now reroute to $f_{-1}$, and vice versa, as shown in figure \ref{fig:operation_B}. This introduces three new crossings in $s$: a crossing between the two rerouted strands, and a crossing between the arc emanating from $f_0$ with each of the two rerouted arcs. These are our ``crossed wires''.

\begin{figure}[ht]
\begin{center}
\begin{tikzpicture}[
scale=1, 
string/.style={thick, draw=red, postaction={nomorepostaction, decorate, decoration={markings, mark=at position 0.5 with {\arrow{>}}}}},
string2/.style={thick, draw=red, postaction={nomorepostaction, decorate, decoration={markings, mark=at position 0.8 with {\arrow{>}}}}},
boundary/.style={ultra thick}]

\draw [boundary] (0,0) -- (0,3);
\draw [string] (0,0.5) -- (0.5,0.5);
\draw [string] (1,1) -- (0,1);
\fill [draw=red, fill=red] (0,1) circle (2pt);
\draw (-0.5,1) node {$f_{-1}$};
\draw [string] (0,1.5) -- (1,1.5);
\fill [draw=red, fill=red] (0,1.5) circle (2pt);
\draw (-0.5,1.5) node {$f_0$};
\draw [string] (1,2) -- (0,2);
\fill [draw=red, fill=red] (0,2) circle (2pt);
\draw (-0.5,2) node {$f_1$};
\draw [string] (0,2.5) -- (0.5,2.5);

\draw [shorten >=1mm, -to, decorate, decoration={snake,amplitude=.4mm, segment length = 2mm, pre=moveto, pre length = 1mm, post length = 2mm}]
(1.5,1.5) -- (2.5,1.5);
\draw (2,2) node {$B$};

\draw [boundary] (4,0) -- (4,3);
\draw [string] (4,0.5) -- (4.5,0.5);
\draw [string2] (5,1) .. controls (4.5,1) and (4.5,2) .. (4,2);
\fill [draw=red, fill=red] (4,1) circle (2pt);
\draw (3.5,1) node {$f_{-1}$};
\draw [string2] (4,1.5) to [bend right=30] (5,1.5);
\fill [draw=red, fill=red] (4,1.5) circle (2pt);
\draw (3.5,1.5) node {$f_0$};
\draw [string2] (5,2) .. controls (4.5,2) and (4.5,1) .. (4,1);
\fill [draw=red, fill=red] (4,2) circle (2pt);
\draw (3.5,2) node {$f_1$};
\draw [string] (4,2.5) -- (4.5,2.5);

\end{tikzpicture}
\caption{The ``crossing wires'' operation $B$.} \label{fig:operation_B}
\end{center}
\end{figure}
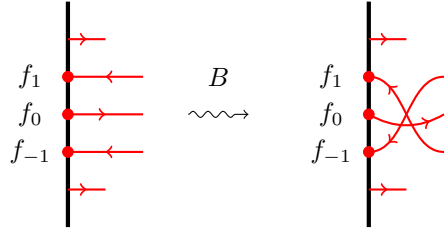

We thus obtain a string diagram $Bs$ well-defined up to regular homotopy. (If we like we could specify the diagram in figure \ref{fig:operation_B} and make $Bs$ defined up to ambient isotopy, but there are two simplest non-ambient-isotopic ways of drawing this arrangement, which are regular homotopic.) In any case $Bs$ is certainly well defined up to spin homotopy or just homotopy. Extending linearly we obtain the maps
\[
\widehat{B} \; : \; \widehat{CS}(\Sigma,F) \To \widehat{CS}(\Sigma,F), \quad
\bar{B} \; : \; CS^\infty(\Sigma,F) \To CS^\infty(\Sigma,F).
\]

Now $\widehat{B}$, $\bar{B}$ are not chain maps, and do not commute with $\partial$. But asking how closely $\widehat{B}$, $\bar{B}$ and $\partial$ commute leads to the lemma, which we now prove.

\begin{proof}
Let $x = \sum_{i=1}^m \bar{s}_i$, where each $\bar{s}_i$ is a distinct spin homotopy class of string diagram without contractible loops; let $s_i$ be a string diagram representing $\bar{s}_i$. Consider $\partial \bar{B} s_i$, which is a sum of diagrams obtained by resolving crossings in $\bar{B} s_i$. There are three diagrams which arise from resolving the three crossings in the crossed wires; these contain all the crossings of $s_i$. The other diagrams in the sum are all the diagrams in $\partial s_i$, with the wires crossed, i.e. $\bar{B} \partial s_i$. Of the first three diagrams, we see that up to homotopy (in fact up to regular homotopy), one is just $s_i$, and the other two both have an outermost non-intersecting strand at $f_0$; hence they are $a_-^* y_i$ and $a_+^* z_i$ for some string diagrams $y_i, z_i$ on $(\Sigma, F')$. See figure \ref{fig:resolving_B_s_i}.

\begin{figure}[ht]
\begin{center}
\begin{tikzpicture}[
scale=1, 
string/.style={thick, draw=red, -to},
boundary/.style={ultra thick}]

\draw (-1,1.5) node {$\partial$};

\draw [boundary] (0,0) -- (0,3);
\draw [rounded corners=8pt] (0,3) -- (2,3) -- (2,0) -- (0,0);
\draw [string] (1,2) .. controls (0.5,2) and (0.5,1) .. (0,1);
\draw [string] (1,1) .. controls (0.5,1) and (0.5,2) .. (0,2);
\draw [string] (0,1.5) to [bend right=30] (1,1.5);
\draw (1.5,2.5) node {$s_i$};

\draw (2.5,1.5) node {$=$};

\draw [xshift=3 cm, boundary] (0,0) -- (0,3);
\draw [xshift=3 cm, rounded corners=8pt] (0,3) -- (2,3) -- (2,0) -- (0,0);
\draw [xshift=3 cm, string] (0,1.5) arc (90:-90:0.25);
\draw [xshift=3 cm, string] (1,1) .. controls (0.5,1) and (0.5,2) .. (0,2);
\draw [xshift=3 cm, string] (1,2) arc (90:270:0.25);
\draw [xshift=3 cm] (1.5,2.5) node {$s_i$};

\draw (5.5,1.5) node {$+$};

\draw [xshift=6 cm, boundary] (0,0) -- (0,3);
\draw [xshift=6 cm, rounded corners=8pt] (0,3) -- (2,3) -- (2,0) -- (0,0);
\draw [xshift=6 cm, string] (0,1.5) arc (-90:90:0.25);
\draw [xshift=6 cm, string] (1,2) .. controls (0.5,2) and (0.5,1) .. (0,1);
\draw [xshift=6 cm, string] (1,1) arc (270:90:0.25);
\draw [xshift=6 cm] (1.5,2.5) node {$s_i$};

\draw (8.5,1.5) node {$+$};

\draw [xshift=9 cm, boundary] (0,0) -- (0,3);
\draw [xshift=9 cm, rounded corners=8pt] (0,3) -- (2,3) -- (2,0) -- (0,0);
\draw [xshift=9 cm, string] (1,1) -- (0,1);
\draw [xshift=9 cm, string] (0,1.5) -- (1,1.5);
\draw [xshift=9 cm, string] (1,2) -- (0,2);
\draw [xshift=9 cm] (1.5,2.5) node {$s_i$};

\draw (11.5,1.5) node {$+$};

\draw [xshift=12 cm, boundary] (0,0) -- (0,3);
\draw [xshift=12 cm, rounded corners=8pt] (0,3) -- (2,3) -- (2,0) -- (0,0);
\draw [xshift=12 cm, string] (1,2) .. controls (0.5,2) and (0.5,1) .. (0,1);
\draw [xshift=12 cm, string] (1,1) .. controls (0.5,1) and (0.5,2) .. (0,2);
\draw [xshift=12 cm, string] (0,1.5) to [bend right=30] (1,1.5);
\draw [xshift=12 cm] (1.5,2.5) node {$\partial s_i$};

\draw (1,-1) node {$\partial B s_i $};
\draw (2.5,-1) node {$=$};
\draw (4,-1) node {$a_-^* y_i$};
\draw (5.5,-1) node {$+$};
\draw (7,-1) node {$a_+^* z_i$};
\draw (8.5,-1) node {$+$};
\draw (10,-1) node {$s_i$};
\draw (11.5,-1) node {$+$};
\draw (13,-1) node {$B \partial s_i$};

\end{tikzpicture}
\caption{Resolutions in $\partial B s_i$.} \label{fig:resolving_B_s_i}
\end{center}
\end{figure}
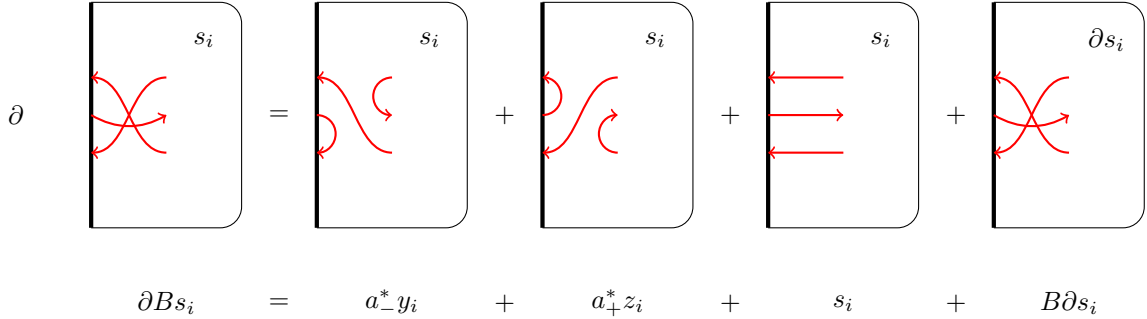

Thus we obtain the following equality in $CS^\infty(\Sigma,F)$, writing $\bar{y}_i, \bar{z}_i$ for the spin homotopy classes of $y_i, z_i$:
\[
\partial \bar{B} \bar{s}_i = \bar{s}_i + a_-^* \bar{y}_i + a_+^* \bar{z}_i + \bar{B} \partial \bar{s}_i.
\]

Summing over $i$ and recalling that $x = \sum_{i=1}^m \bar{s}_i$ gives
\[
\partial \bar{B} x = x + a_-^* \left( \sum_{i=1}^m \bar{y}_i \right)
+ a_+^* \left( \sum_{i=1}^m \bar{z}_i \right) + \bar{B} \partial x.
\]
Recalling that $\partial x = 0$, that as always we are working mod $2$, and setting $u = \bar{B}x$, $y = \sum_{i=1}^m \bar{y}_i$, $z = \sum_{i=1}^m \bar{z}_i$ gives the desired equality $x = a_-^* y + a_+^* z + \partial u$. It remains only to show that $\partial y = \partial z = 0$. Applying $\partial$ to this equality, and recalling that creation operators are chain maps, gives
\[
a_-^* \sum_{i=1}^m \partial \bar{y}_i = a_+^* \sum_{i=1}^m \partial \bar{z}_i.
\]
Both sides are sums of (spin homotopy classes of) string diagrams, but on the left all diagrams have a non-intersecting arc connecting $f_0$ to $f_{-1}$; while on the right all diagrams have a non-intersecting arc connecting $f_0$ to $f_1$. Thus no diagram which occurs on the left is homotopic to any diagram which occurs on the right; so both sides must be $0$. Thus $a_-^* \partial y = a_+^* \partial z = 0$. As creation operators are injective we have $\partial y = \partial z = 0$.
\end{proof}

By the same proof we obtain the corresponding result for $\widehat{CS}$; the same proof works whether we consider diagrams up to spin homotopy or just homotopy.
Let $(\Sigma,F)$, $(\Sigma,F')$ be as above, and let $a_-^*, a_+^*$ now be creation operators $\widehat{CS}(\Sigma,F') \To \widehat{CS}(\Sigma,F)$ obtained by inserting strings in the same places as above.
\begin{lem}[Crossed wires lemma, $\widehat{CS}$ version]
Suppose $x \in \widehat{CS}(\Sigma,F)$ satisfies $\partial x = 0$. Then there exist $y,z \in \widehat{CS}(\Sigma,F')$ and $u \in \widehat{CS}(\Sigma,F)$ such that
\[
\partial y = \partial z = 0
\quad \text{and} \quad
x = a_-^* y + a_+^* z + \partial u.
\]
\qed
\end{lem}

\subsection{Inductive step}

The crossed wires lemma now allows us to find a basis for each $\widehat{HS}(D^2, F_n)$.

\begin{prop}
\label{prop:induction_step}
Let $n \geq 1$. If the $2^{n-1}$ elements of the form $v_w$ in $\widehat{HS}(D^2, F_n)$ form a basis, then the $2^n$ elements of the form $v_w$ in $\widehat{HS}(D^2, F_{n+1})$ also form a basis.
\end{prop}

\begin{proof}
An element of $\widehat{HS}(D^2, F_{n+1})$ is represented by $x \in \ker \partial \subseteq \widehat{CS}(D^2, F_{n+1})$, i.e. $x$ is a linear combination of (homotopy classes of) string diagrams without loops on $(D^2, F_{n+1})$ such that $\partial x = 0$. 

From the crossed wires lemma, there exist $y,z \in \widehat{CS}(D^2, F_n)$ and $u \in \widehat{CS}(D^2, F_{n+1})$ such that
\[
\partial y = \partial z = 0, \quad x = a_-^* y + a_+^* z + \partial u.
\]
It follows that the homology class of $x$ lies in $a_-^* \widehat{HS} (D^2, F_n) + a_+^* \widehat{HS} (D^2, F_n)$. As $\widehat{HS}(D^2, F_n)$ is spanned by the $v_w$ for words $w$ of length $n-1$, $a_-^* \widehat{HS}(D^2, F_n)$ is spanned by the $a_-^* v_w = v_{-w}$ and $a_+^* \widehat{HS}(D^2, F_n)$ is spanned by the $a_+^* v_w = v_{+w}$; hence $\widehat{HS}(D^2, F_{n+1})$ is spanned by the $v_w$ for words of length $n$. Lemma \ref{lem:linearly_independent} says the $v_w$ are linearly independent, so they form a basis.
\end{proof}

As $\{v_\emptyset\}$ forms a basis for $\widehat{HS}(D^2, F_1)$ by lemma \ref{lem:base_case}, proposition \ref{prop:induction_step} immediately gives:
\begin{cor}
For all $n \geq 0$, the elements $v_w$, for words of length $n$ form a basis of $\widehat{HS}(D^2, F_{n+1})$.
\qed
\end{cor}

\begin{cor}
Any nonzero homology class in $\widehat{HS}(D^2, F_n)$ is represented by a linear combination of string diagrams without intersections, in fact which are sets of sutures.
\qed
\end{cor}

Indeed, writing $\widehat{CS}_{sut}$ for the subspace of $\widehat{CS}$ generated by sets of sutures, and $\partial$ for the differential on $\widehat{CS}$, we have a linear map
\[
i \; : \; \widehat{CS}_{sut}(D^2, F_n) \To \ker \partial \To \frac{\ker \partial}{\IIm \partial} = \widehat{HS}(D^2, F_n)
\]
arising from the composition of inclusion and quotient maps. Since, as we have shown, $\widehat{HS}(D^2, F_n)$ is spanned by homology classes of (homotopy classes of) string diagrams which are sutures, this map $i$ is surjective.

Now the sum of a bypass triple of sutures on $(D^2, F_n)$ lies in the image of $\partial$, as shown in figure \ref{fig:fundamental_boundary}. If we introduce ``crossed wires'' at the site of the bypass, then $\partial$ gives precisely the sum of the diagrams in the bypass triple. Thus the surjective map $i$ factors as
\[
\widehat{CS}_{sut}(D^2, F_n) \To \frac{\widehat{CS}_{sut}(D^2,F_n)}{\widehat{Byp}(D^2, F_n)} \To \widehat{HS}(D^2, F_n).
\]
In the notation of \cite{Me09Paper}, the quotient $\frac{\widehat{CS}_{sut}(D^2,F_n)}{\widehat{\Byp}(D^2,F_n)}$ is $SFH_{comb}(T,n)$, which is computed in that paper to have dimension $2^{n-1}$. Above we computed that $\widehat{HS}(D^2, F_n)$ has the same dimension. Thus we have an isomorphism
\[
\widehat{HS}(D^2, F_n) \cong \frac{\widehat{CS}_{sut}(D^2, F_n)}{\widehat{\Byp}(D^2,F_n)}.
\]
As discussed in section \ref{sec:floer-theoretic}, from \cite{Me09Paper}, sutures modulo bypasses on $(D^2, F_n)$ gives $SFH(D^2 \times S^1, F_n \times S^1)$. So the above isomorphic vector spaces are also isomorphic with $SFH(D^2 \times S^1, F_n \times S^1)$. From \cite{Me12_itsy_bitsy}, this is also isomorphic to $(\Z_2 \0 \oplus \Z_2 \1 )^{\otimes (n-1)}$.

As bypass surgery preserves Euler class, we can immediately restrict to sutures of a specific Euler class and obtain an isomorphism 
\[
\widehat{HS}_e(D^2, F_n) \cong \frac{\widehat{CS}_{sut,e}(D^2,F_n)}{\widehat{Byp_e}(D^2,F_n)}.
\]
As discussed in \cite{Me09Paper}, this is also isomorphic to a summand $SFH_e(D^2 \times S^1, F_n \times S^1)$ of $SFH(D^2 \times S^1, F_n \times S^1)$. And as discussed in \cite{Me12_itsy_bitsy}, this is also isomorphic to the summand of $(\Z_2 \0 \oplus \Z_2 \1)^{\otimes (n-1)}$ generated by tensor products $e_1 \otimes \cdots \otimes e_n$ where each $e_i \in \{\0, \1\}$ and the number of $\1$'s minus $\0$'s is $e$.

This proves theorem \ref{thm:main_thm}.

\section{Discs with spin}
\label{sec:discs_with_spin}

We now prove the main theorem \ref{thm:main_theorem_infty} for $HS^\infty$ when $\Sigma=D^2$. The proof runs along the lines of the proof for $\widehat{HS}$; most of the effort goes into the base case.

\subsection{Base case}
\label{sec:spin_base_case}

We analyse string diagrams $s$ on $(D^2, F_1)$ up to spin homotopy. Obviously any such $s$ is homotopic to a single strand running between the points of $F_1$, together with some number $m \geq 0$ of closed curves. The homotopy classes of string diagrams are parametrised by $m$.

Given a string diagram $s$, we may perform a string Reidemeister I move, adding a clockwise or anticlockwise whirl on it, which adjusts the generalised Euler class by $-2$ or $2$ respectively. By the definition of generalised Euler class, $e(s)$ must be even: all the curves of $s$ have curvature which is an even multiple of $2\pi$. By lemma \ref{prop:spin_homotopy}, the spin homotopy class of a string diagram is determined by its homotopy class (i.e. $m$) and its generalised Euler class $e$. 

Thus, the spin homotopy classes of string diagrams on $(D^2, F_1)$ are precisely parametrised by pairs of integers $(m,e)$ where $m \geq 0$ and $e$ is even. Let $\sigma_{m,e}$ denote this spin homotopy class.

We noted in section \ref{sec:generalised_euler_class} that $\partial$ preserves $e$. As $\partial$ is well-defined on spin homotopy classes, to compute $\partial \sigma_{m,e}$ it's sufficient to take a single representative string diagram $s_{m,e}$ of the class $\sigma_{m,e}$. We can take $s_{m,e}$ to consist of $m$ non-intersecting anticlockwise closed curves, and a strand which has some number $k$ of whirls (and hence $|k|$ self-intersections) added to obtain the correct $e$. See figure \ref{fig:sigma_m_e}. We can easily compute $k = \frac{e}{2} - m$.

\begin{figure}[ht]
\begin{center}

\begin{tikzpicture}[
scale=1.5, 
string/.style={thick, draw=red, postaction={nomorepostaction, decorate, decoration={markings, mark=at position 0.5 with {\arrow{>}}}}}]

\draw (0,0) circle (1 cm);

\draw [string] (0,-1) .. controls (0,-0.5) and (-0.5,0) .. (-0.5,-0.5) .. controls (-0.5,-1) and (0,-0.5) .. (0,0) .. controls (0,0.5) and (-0.5,1) .. (-0.5,0.5) .. controls (-0.5,0) and (0,0.5) .. (0,1);
\draw [string] (0.75,-0.25) circle (0.125 cm);
\draw [string] (0.75,0.25) circle (0.125 cm);

\end{tikzpicture}
\caption{The string diagram $s_{m,e}$. There are $k$ whirls in the arc, and $m$ anticlockwise closed curves.}
\label{fig:sigma_m_e}
\end{center}
\end{figure}
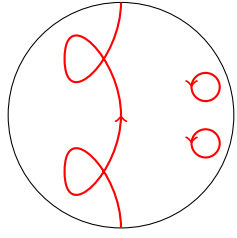

We see that, if $k$ is even, then $\partial s_{m,e}$ consists of an even number of spin homotopic diagrams; while if $k$ is odd, then $\partial s_{m,e}$ consists of an odd number of spin homotopic diagrams, with $m+1$ closed curves and generalised Euler class $e$. Hence we have proved the following lemma.
\begin{lem}
The chain complex $CS^\infty(D^2, F_1)$ is freely generated over $\Z_2$ by $\{ \sigma_{m,e} \}$, over all integers $m,e$ satisfying $m \geq 0$ and $e$ even. The differential is given by
\[
\partial \sigma_{m,e} = \left\{ \begin{array}{rl} 
	0 & \text{$m + \frac{e}{2}$ even} \\
	\sigma_{m+1,e} & \text{$m + \frac{e}{2}$ odd}
	\end{array} \right.
\]
\qed
\end{lem}

Since $\partial$ preserves $e$, the chain complex and homology split into summands $CS^\infty_e (D^2, F_1)$ and $HS^\infty_e (D^2, F_1)$ over all even $e \in \Z$. When $e$ is not a multiple of $4$, i.e. $e = 4i+2$, the differential is given by
\[
\sigma_{0,e} \mapsto \sigma_{1,e}, \quad \sigma_{2,e} \mapsto \sigma_{3,e}, \quad \ldots
\]
and the homology in this summand is trivial. When $e$ is a multiple of $4$, the differential is given by
\[
\sigma_{0,e} \mapsto 0, \quad \sigma_{1,e} \mapsto \sigma_{2,e}, \quad \sigma_{3,e} \mapsto \sigma_{4,e}, \ldots
\]
and the homology is generated by (the homology class of) $\sigma_{0,e}$.

Thus $HS^\infty(D^2, F_1)$ has basis given by (homology classes of) $\sigma_{0,e}$, over all $e \in 4 \Z$. Recalling the definition of the $U$ map (section \ref{sec:defn_of_U}) and its effect on Euler class (section \ref{sec:gen_Euler_class}), we have $\sigma_{0,4i} = U^i \sigma_{0,0}$. We then immediately have the following computation.
\begin{prop}
The map
\[
\Z_2[U,U^{-1}] \To HS^\infty(D^2, F_1)
\]
which takes $1 \mapsto \sigma_{0,0}$ and preserves the action of $U$ defined on $HS^\infty(D^2, F_1)$, is an isomorphism of $\Z_2[U,U^{-1}]$-modules.
\end{prop}
Note $1 \in \Z_2[U,U^{-1}]$ corresponds to the vacuum diagram $v_\emptyset$, and $U^j \in \Z_2[U,U^{-1}]$ corresponds to that diagram with $j$ whirls, a ``whirly vacuum''. ``Homology is generated by whirly vacua''.

Note also how closed curves have disappeared in the homology. Any $x \in HS^\infty (D^2, F_1)$ can be written as $x = \sum_i \sigma_i$, where each $\sigma_i$ is the spin homotopy class of a string diagram with no closed curves. Each $\sigma_i$ can be taken to be a whirly vacuum.

\subsection{Inductive step}

Annihilation and creation operators can then be applied on each $(D^2, F_n)$ as in section \ref{sec:building_a_basis}, giving maps
\[
a_\pm : HS^\infty (D^2, F_n) \To HS^\infty(D^2, F_{n-1}), \quad
a_\pm^* : HS^\infty(D^2, F_n) \To HS^\infty(D^2, F_{n+1})
\]
which satisfy similar relations. They also commute with $U$ and hence give maps of $\Z_2[U,U^{-1}]$-modules. Again each $a_\pm^*$ is injective. And again for a word $w$ of length $n$ on $\{-,+\}$ we obtain a composite creation operator $a_w^*$ and let $v_w = a_w^* v_\emptyset = a_w^* 1$.

\begin{lem}
The $2^n$ elements of the form $v_w$ in $HS^\infty(D^2, F_{n+1})$ are linearly independent over $\Z_2[U,U^{-1}]$.
\end{lem}

\begin{proof}
Again for each word $w$ there is a sequence of annihilations which send $v_w \mapsto v_\emptyset$ but each other $v_{w'} \mapsto 0$. These annihilations commute with $U$. If we have a nontrivial linear combination over $\Z_2[U,U^{-1}]$, then
\[
\sum_i p_{w_i}(U) v_{w_i} = 0
\]
where each $p_{w_i}(U)$ is a Laurent polynomial in $U$ over $\Z_2$. Applying the sequence of annihilations for $w_i$ then sends $p_{w_i}(U) v_{w_i} \mapsto p_{w_i}(U)$ but every other term to $0$. Hence $p_{w_i} (U) = 0$.
\end{proof}

On the other hand, if the $2^{n-1}$ elements of the form $v_w$ generate $HS^\infty(D^2, F_n)$ over $\Z_2[U,U^{-1}]$, then the crossed wires lemma \ref{lem:decomposition} shows that $HS^\infty(D^2, F_{n+1})$ is spanned by $a_-^* HS^\infty(D^2, F_n)$ and $a_+^* HS^\infty(D^2, F_n)$, hence is generated by the $v_w$ for words $w$ of length $n$.

\begin{cor}
For all $n \geq 0$, the elements $v_w$ for words of length $n$ form a basis of $HS^\infty(D^2, F_{n+1})$ over $\Z_2[U,U^{-1}]$. The elements $U^j v_w$, over all $j \in \Z$ and words $w$ of length $n$, form a basis of $HS^\infty(D^2, F_{n+1})$ over $\Z_2$.
\qed
\end{cor}

In particular,
\[
HS^\infty(D^2, F_n) \cong \Z_2[U,U^{-1}] \otimes \widehat{HS}(D^2, F_n) \cong \Z_2[U,U^{-1}] \otimes \left( \frac{ \widehat{CS}_{sut} (D^2, F_n) }{ \widehat{\Byp}(D^2, F_n) } \right),
\]
and theorem \ref{thm:rough_main_thm1} is proved.

For the remaining details of theorem \ref{thm:main_theorem_infty}, every element of $HS^\infty(D^2, F_n)$ is a $\Z_2[U,U^{-1}]$-linear combination of string diagrams which are sets of sutures. So the composition
\[
CS^\infty_{sut} (D^2, F_n) \To \frac{CS^\infty_{sut}(D^2,F_n)}{\Byp^\infty(D^2,F_n)} \To HS^\infty(D^2,F_n)
\]
is surjective, and comparing dimensions, we have
\[
HS^\infty(D^2, F_n) \cong \frac{ CS^\infty_{sut}(D^2, F_n)}{\Byp^\infty(D^2,F_n)}
\]
as $\Z_2[U,U^{-1}]$-modules. Decomposing over powers of $U$, we obtain part (ii) of theorem \ref{thm:main_theorem_infty}, completing the proof.

\addcontentsline{toc}{section}{References}

\small

\bibliography{danbib}
\bibliographystyle{amsplain}

\end{document}